\documentclass[11pt,a4paper,reqno]{amsart}

\title{Decomposition of tensor products of Demazure crystals}
\author{Takafumi Kouno}
\date{}
\address{Department of Mathematics, Tokyo Institute of Technology, 2-12-1 Oh-okayama, Meguro-ku, Tokyo 152-8551, Japan}
\email{kouno.t.ab@m.titech.ac.jp}

\usepackage{amsmath}
\usepackage{amsthm}
\usepackage{amsfonts}
\usepackage{amssymb}
\usepackage{enumerate}
\usepackage[pdftex]{graphicx}
\usepackage{tikz}
\usetikzlibrary{cd}
\usepackage{color}

\newcommand{\wt}{\mathrm{wt}}
\newcommand{\ch}{\mathrm{ch}}

\newcommand{\dist}{\mathrm{dist}}

\theoremstyle{definition}
\newtheorem{thm}{Theorem}[section]
\newtheorem{defn}[thm]{Definition}
\newtheorem{lem}[thm]{Lemma}
\newtheorem{prop}[thm]{Proposition}
\newtheorem{cor}[thm]{Corollary}
\newtheorem{exm}[thm]{Example}

\theoremstyle{remark}
\newtheorem{rem}[thm]{Remark}

\allowdisplaybreaks[4]

\makeatletter
	
	\@addtoreset{equation}{section}
\makeatother

\begin{document}

\begin{abstract}
In general, a tensor product of Demazure crystals does not decompose into a disjoint union of Demazure crystals. However, under a certain condition, a tensor product decomposes into a disjoint union of Demazure crystals. In this paper, we introduce a necessary and sufficient condition for every connected component of a tensor product of two Demazure crystals to be isomorphic to some Demazure crystal. Moreover, we consider a recursive formula describing connected components of tensor products of arbitrary Demazure crystals. As an application, we discuss the key positivity problem, which is the problem whether a product of key polynomials is a linear combination of key polynomials with nonnegative integer coefficients or not. Also, we obtain a crystal-theoretic analog of the Leibniz rule for Demazure operators.
\end{abstract}

\maketitle

\section{Introduction}

Let $\mathfrak{g}$ be a finite dimensional simple Lie algebra over $\mathbb{C}$, and $U_q(\mathfrak{g})$ the quantum group of $\mathfrak{g}$ over the field $\mathbb{Q}(q)$ of rational functions over $\mathbb{Q}$ in a variable $q$. For a dominant integral weight $\lambda$, we denote by $L(\lambda)$ the irreducible highest weight $U_q(\mathfrak{g})$-module with highest weight $\lambda$. It is well-known that the category of finite dimensional $U_q(\mathfrak{g})$-modules of type {\boldmath $1$}  is semisimple. Hence, for dominant integral weights $\lambda$ and $\mu$, $L(\lambda) \otimes L(\mu)$ decomposes into a direct sum of finite dimensional type {\boldmath $1$} irreducible highest weight $U_q(\mathfrak{g})$-modules, namely, there exists a nonnegative integer $c_{\lambda, \mu}^\nu$ for each dominant integral weight $\nu$ such that
\begin{align*}
L(\lambda) \otimes L(\mu) \simeq \bigoplus_{\nu \in P^+}L(\nu)^{\oplus c_{\lambda, \mu}^\nu},
\end{align*}
where $P^+$ is the set of dominant integral weights. The integers $c_{\lambda, \mu}^\nu$ are called Littlewood-Richardson coefficients. 

For a dominant integral weight $\nu$, the module $L(\nu)$ has a crystal basis $(\mathcal{L}(\nu), \mathcal{B}(\nu))$. From the decomposition above of $L(\lambda) \otimes L(\mu)$, we obtain the corresponding formula
\begin{align*}
\mathcal{B}(\lambda) \otimes \mathcal{B}(\mu) \simeq \bigsqcup_{\nu \in P^+} \mathcal{B}(\nu)^{\oplus c_{\lambda, \mu}^\nu};
\end{align*}
namely, a tensor product of crystal bases decomposes into a disjoint union of connected highest weight crystals.

Now we consider Demazure crystals. Does a tensor product of Demazure crystals decompose into a disjoint union of Demazure crystals? In general, the answer is no. However, in some special cases, it is known that a tensor product of Demazure crystals decomposes into a disjoint union of Demazure crystals. Let $W$ be the Weyl group of $\mathfrak{g}$. For $w \in W$ and $\lambda \in P^+$, we denote by $\mathcal{B}_w(\lambda)$ the Demazure crystal of lowest weight $w\lambda$. Lakshmibai, Littelmann, and Magyar proved the following theorem.

\begin{thm}[{\cite[Proposition 12]{Magyar}}]
Let $\lambda$ and $\mu$ be dominant integral weights, and $w$ an element of $W$. Then $\mathcal{B}_e(\lambda) \otimes \mathcal{B}_w(\mu)$ decomposes into a disjoint union of Demazure crystals.
\end{thm}
Note that $\mathcal{B}_e(\lambda) = \{ b_\lambda \}$, where $b_\lambda \in \mathcal{B}(\lambda)$ is the highest weight element.

The first main result of this paper is the following theorem, which generalizes the above one.
Let $\Phi^+$ be the set of positive roots, $\{\alpha_i \}_{i \in I}$ the set of simple roots, and $s_i$ the simple reflection corresponding to $\alpha_i$. We set $W_\lambda := \{ w \in W \ | \ w\lambda = \lambda \}$ for $\lambda \in P^+$. For $w \in W$, we denote by $\lfloor w \rfloor^\lambda$ the minimal-length representative for the coset $wW_\lambda$, and by $\lceil w \rceil^\lambda$ the maximal-length representative for the coset $wW_\lambda$. For $w \in W$, we denote by $\ell(w)$ the length of $w$.
\begin{thm}\label{mainthm}
Let $\lambda$ and $\mu$ be dominant integral weights. For $w \in W$, we define the group $W_w$ as the subgroup of $W$ generated by $\{ s_i \ | \ i \in I, \ \ell(s_{i}w) < \ell(w) \}$. Then, for $v, w \in W$, the set $\mathcal{B}_v(\lambda) \otimes \mathcal{B}_w(\mu)$ decomposes into a disjoint union of Demazure crystals if and only if $\lfloor v \rfloor^\lambda \in W_{\lceil w \rceil^\mu}$.
\end{thm}

The key positivity problem is the problem whether a product of key polynomials is a linear combination of key polynomials with nonnegative integer coefficients or not. If $\mathfrak{g}$ is of type A, then the character of a Demazure crystal is identical to a key polynomial. Hence the above theorem tells us a sufficient condition for the key positivity problem. For $w \in W$ and $\lambda \in P^+$, we denote by $\kappa_{w\lambda}$ the character of $\mathcal{B}_w(\lambda)$. Then we obtain the following theorem. 

\begin{thm}\label{app1}
Let $\lambda$ and $\mu$ be dominant integral weights, $v$ and $w$ elements of $W$. If $\lfloor v \rfloor^\lambda \in W_{\lceil w \rceil^\mu}$ or $\lfloor w \rfloor^\mu \in W_{\lceil v \rceil^\lambda}$, then the product $\kappa_{v\lambda} \kappa_{w\mu}$ is a linear combination of the characters of Demazure crystals with nonnegative integer coefficients.
\end{thm}

The next main result of this paper is a recursive formula describing connected components of tensor products of arbitrary Demazure crystals. For each dominant integral weight $\lambda$, we identify $\mathcal{B}(\lambda)$ as the crystal composed of all Lakshmibai-Seshadri paths of shape $\lambda$. Let $\lambda$ and $\mu$ be dominant integral weights, $v$ and $w$ elements of $W$. For a $\lambda$-dominant path $\pi \in \mathcal{B}_w(\mu)$, we denote by $C(\pi, v)$ the connected component of $\mathcal{B}_v(\lambda) \otimes \mathcal{B}_w(\mu)$ whose highest weight element is $\pi^\lambda \otimes \pi$, where $\pi^\lambda$ is the highest weight element of $\mathcal{B}(\lambda)$. An idea is to compare $C(\pi, v)$ and $C(\pi, s_i v)$ for $i \in I$ with $\ell(s_i v) > \ell (v)$. We obtain the following formula.

\begin{thm}\label{mainthm2}
Let $\lambda, \mu$ be dominant integral weights, $v, w$ elements of the Weyl group $W$. Take $i \in I$ which satisfies $\ell (s_i v) > \ell (v)$. Also, we take a $\lambda$-dominant path $\pi \in \mathcal{B}_w(\mu)$. Then we have the following equation.
\begin{align*}
& \ C(\pi, s_i v) \\
=& \ \left( \bigcup_{a \geq 0} \tilde{f}_i^a(C(\pi, v)) \setminus \{ 0 \} \right) \\
& \ \setminus \left\{ \tilde{f}_{i}^{\langle \wt(\pi_1), \alpha_i^\vee \rangle} (\pi_1) \otimes \tilde{f}_{i}^b(\pi_2) \left| \begin{array}{l}\pi_1 \otimes \pi_2 \in C(\pi, v), \\ \tilde{e}_{i}(\pi_1) = 0, \ \tilde{f}_{i}(\pi_2) \not\in \mathcal{B}_w(\mu) \sqcup \{ 0 \}, \\ 1 \leq b \leq \langle \wt(\pi_2), \alpha_i^\vee \rangle. \end{array} \right. \right\}.
\end{align*}
\end{thm}

As an application of this theorem, one can prove the following relation, which is an analog of the Leibniz rule for Demazure operators. For $i \in I$ and $\lambda \in P^+$, the map $S_i : \mathcal{B}(\lambda) \to \mathcal{B}(\lambda)$ is defined as follows.
\begin{align*}
S_i : \mathcal{B}(\lambda) \to \mathcal{B}(\lambda), \ b \mapsto \left\{ \begin{array}{ll} \tilde{f}_i^{\langle \wt(b), \alpha_i^\vee \rangle}(b) & \text{if } \langle \wt(b), \alpha_i^\vee \rangle \geq 0, \\ \tilde{e}_i^{-\langle \wt(b), \alpha_i^\vee \rangle}(b) & \text{if } \langle \wt(b), \alpha_i^\vee \rangle \leq 0. \end{array} \right.
\end{align*}
Also, we define the map $\tilde{e}_i^{\max} : \mathcal{B}(\lambda) \to \mathcal{B}(\lambda)$ for $i \in I$ and $\lambda \in \mathcal{B}(\lambda)$ as $\tilde{e}_i^{\max} (b) := \tilde{e}_i^{\varepsilon_i(b)}(b)$ for $b \in \mathcal{B}(\lambda)$.

\begin{thm}\label{app2}
Let $\lambda, \mu$ be dominant integral weights, $v, w$ elements of $W$. Take $i \in I$ which satisfies $\ell(s_i v) > \ell (v)$. Then it holds that
\begin{align*}
& \ \bigcup_{n \geq 0} \tilde{f}_{i}^{n} (\mathcal{B}_{v}(\lambda) \otimes \mathcal{B}_{w}(\mu)) \setminus \{ 0 \} \\
 = & \ (\mathcal{B}_{s_i v}(\lambda) \otimes \mathcal{B}_{w}(\mu)) \sqcup (S_{i} (\tilde{e}_i^{\max}(\mathcal{B}_v(\lambda))) \otimes (\mathcal{B}_{s_i w}(\mu) \setminus \mathcal{B}_{w}(\mu))).
\end{align*}

\end{thm}

This paper is organized as follows. In Section \ref{quantum_group}, we recall some basic facts about quantum groups, their representations, and crystal bases. In Section \ref{Weyl}, we give some properties of Weyl groups. In Section \ref{crystal}, we recall the definition of crystals and review the definition and properties of Demazure crystals. In Section \ref{LS}, we introduce Lakshmibai-Seshadri paths. In Section \ref{mainsec}, we give some examples of tensor products of Demazure crystals, then prove Theorem \ref{mainthm}. In section \ref{recursion}, we consider Theorem \ref{mainthm2}, which is a recursive formula describing connected components of tensor products of Demazure crystals. In section \ref{application}, we give applications of our theorems, first, we discuss Theorem \ref{app1} and the key positivity problem, and next we consider Theorem \ref{app2}, which is an analog of the Leibniz rule for Demazure operators.

\subsection*{Acknowledgments} The author is deeply indebted to Satoshi Naito for numerous helpful suggestions and in-depth discussions. He would like to thank Daisuke Sagaki for many useful comments. He is grateful to Fumihiko Nomoto, Naoki Fujita, and Hideya Watanabe for various valuable pieces of advice.

\section{Quantum Groups and Their Representations}\label{quantum_group}
In this section, we recall some basic facts about quantum groups and their representation theory. For details, see \cite{Jantzen}.

\subsection{Quantum groups}
Let $\mathfrak{g}$ be a finite dimensional simple Lie algebra over $\mathbb{C}$, $\mathfrak{h}$ the Cartan subalgebra of $\mathfrak{g}$, $\Phi$ the root system of $\mathfrak{g}$, $\Phi^+$ the set of positive roots, $\Pi = \{ \alpha_i \}_{i \in I}$ the set of simple roots, $\Pi^\vee = \{ \alpha_i^\vee \}_{i \in I}$ the set of simple coroots, $s_i$ the simple reflection corresponding to $\alpha_i$, $i \in I$, $W := \langle s_i \ | \ i \in I \rangle$ the Weyl group of $\mathfrak{g}$. Then there is a $W$-invariant inner product $( \ , \ )$ on the vector space over $\mathbb{R}$ generated by $\Phi$ such that $(\alpha, \alpha) = 2$ for all short roots $\alpha \in \Phi$. Also, let $P$ be the weight lattice of $\mathfrak{g}$, $P^+$ the set of dominant integral weights, $\mathbb{Z}\Phi$ the root lattice of $\mathfrak{g}$. We set $k := \mathbb{Q}(q)$ be the field of rational functions over $\mathbb{Q}$ in a variable $q$.

For $i \in I$ and $a \in \mathbb{Z}$, we set $[a]_i := (q_i^a - q_i^{-a})/(q_i - q_i^{-1})$, where $q_i := q^{(\alpha_i, \alpha_i)/2}$. Then we can define the $q$-binomial coefficients by 
\begin{align*}
\begin{bmatrix} a \\ n \end{bmatrix}_i := \frac{[a]_i [a-1]_i \cdots [a-n+1]_i}{[n]_i [n-1]_i \cdots [1]_i}
\end{align*}
for $i \in I$ and $a, n \in \mathbb{Z}$ with $n >0$; we set
\begin{align*}
\begin{bmatrix} a \\ 0 \end{bmatrix}_i := 1
\end{align*}
for $i \in I$ and $a \in \mathbb{Z}$. In addition, we can define the $q$-analog of the factorial for all $n \in \mathbb{Z}$ with $n > 0$ by
\begin{align*}
[n]!_i := [n]_i [n-1]_i \cdots [1]_i
\end{align*}
for $i \in I$; we set $[0]!_i := 1$.

\begin{defn}[{\cite[Section 4.3]{Jantzen}}]
The {\it quantum group} $U_q(\mathfrak{g})$ is the associative algebra over $k$ with generators $E_i, F_i, K_i, K_i^{-1}$, $i \in I$, and relations
\begin{align*}
\begin{array}{ll}
K_i K_i^{-1} = K_i^{-1} K_i = 1 & (i \in I), \\ \\
K_i K_j = K_j K_i & (i, j \in I), \\ \\
K_i E_j K_i^{-1} = q^{(\alpha_i, \alpha_j)} E_j & (i, j \in I), \\ \\
K_i F_j K_i^{-1} = q^{-(\alpha_i, \alpha_j)} F_j & (i, j \in I), \\ \\
E_i F_j - F_j E_i = \delta_{ij} \dfrac{K_i - K_i^{-1}}{q_i - q_i^{-1}} & (i, j \in I), \\ \\
\displaystyle\sum_{s = 0}^{1-a_{ij}} (-1)^s \begin{bmatrix} 1-a_{ij} \\ s \end{bmatrix}_i E_i^{1-a_{ij}-s} E_j E_i^s = 0 & (i, j \in I), \\ \\
\displaystyle\sum_{s = 0}^{1-a_{ij}} (-1)^s \begin{bmatrix} 1-a_{ij} \\ s \end{bmatrix}_i F_i^{1-a_{ij}-s} F_j F_i^s = 0 & (i, j \in I),
\end{array}
\end{align*}
where $\delta_{ij}$ is the Kronecker delta, and $a_{ij} = 2(\alpha_i, \alpha_j)/(\alpha_i, \alpha_i)$ is the $(i, j)$-entry of the Cartan matrix of $\mathfrak{g}$.
\end{defn}

For each $i \in I$, the subalgebra of $U_q(\mathfrak{g})$ generated by $E_i, F_i, K_i, K_i^{-1}$ is isomorphic to $U_{q_i}(\mathfrak{sl}_2)$; we denote this algebra by $U^i$.

\subsection{Representations of quantum groups}
In this subsection, we will survey the representation theory of quantum groups.

Let $M$ be a finite dimensional $U_q(\mathfrak{g})$-module. For $\lambda \in P$, we set $M_\lambda := \{ m \in M \ | \ K_i m = q^{(\lambda, \alpha_i)}m \text{ for all } i \in I \}$. We call $M$ {\it type {\boldmath $1$}} if $M = \bigoplus_{\lambda \in P} M_\lambda$ holds. From now on, finite dimensional $U_q(\mathfrak{g})$-modules are assumed to be of type {\boldmath $1$}.

\begin{defn}[{\cite[Section 5.5]{Jantzen}}]
For $\lambda \in P$, we define the {\it Verma module} $M(\lambda)$ by
\begin{align*}
M(\lambda) := U_q(\mathfrak{g}) \left/ \left( \sum_{i \in I} U_q(\mathfrak{g}) \left( K_i - q^{(\lambda, \alpha_i)} \right) + \sum_{i \in I} U_q(\mathfrak{g}) E_i \right) \right. .
\end{align*}
\end{defn}

The Verma module $M(\lambda)$ has a unique irreducible quotient. We denote this module by $L(\lambda)$.

\begin{thm}[{\cite[Theorem 5.10]{Jantzen}}]\label{fdmod}
For $\lambda \in P^+$, $L(\lambda)$ is finite dimensional. Conversely, each finite dimensional irreducible $U_q(\mathfrak{g})$-module is isomorphic to $L(\lambda)$ for some $\lambda \in P^+$.
\end{thm}

It is known that every finite dimensional $U_q(\mathfrak{g})$-module decomposes into a direct sum of finite dimensional irreducible $U_q(\mathfrak{g})$-modules. Namely, we know the following.

\begin{thm}[{\cite[Theorem 5.17]{Jantzen}}]\label{ss}
The category of finite dimensional $U_q(\mathfrak{g})$-modules is semisimple.
\end{thm}

Take $\lambda \in P^+$ arbitrarily. There is a vector $v_\lambda$ in $L(\lambda)_\lambda \setminus \{ 0 \}$ such that $E_i v_\lambda$ is equal to $0$ for all $i \in I$. We call this $v_\lambda$ a {\it highest weight vector} of $L(\lambda)$. It is known that $L(\lambda)$ is generated by $v_\lambda$ over $U_q(\mathfrak{g})$.

\subsection{Crystal bases}
In the previous subsection, we introduced the finite dimensional irreducible $U_q(\mathfrak{g})$-module $L(\lambda)$ for $\lambda \in P^+$. In fact, these modules have certain good bases, called crystal bases. In this subsection, we construct these bases.

Let $M$ be a finite dimensional $U_q(\mathfrak{g})$-module. Take $i \in I$. Recall that every finite dimensional $U_{q_i}(\mathfrak{sl}_2)$-module decomposes into a direct sum of finite dimensional irreducible $U_{q_i}(\mathfrak{sl}_2)$-modules by Theorem \ref{ss}. Therefore, by considering $M$ as a $U^i$-module, we see that for $\lambda \in P$, every $x \in M_\lambda$ can be uniquely written as:
\begin{align*}
x = \sum_{j \geq 0, \ j \geq -\langle \lambda, \alpha_i^\vee \rangle} \frac{1}{[j]!_i}F_i^j x_j,
\end{align*}
where $x_j \in M_{\lambda + j \alpha_i}$ and $E_i x_j = 0$ for all $j$. Then we define the operators $\tilde{f}_i$ and $\tilde{e}_i$ by
\begin{align*}
\tilde{f}_i (x) &:= \sum_{j \geq 0, \ j \geq -\langle \lambda, \alpha_i^\vee \rangle} \frac{1}{[j+1]!_i}F_i^{j+1} x_j, \\
\tilde{e}_i (x) &:= \sum_{j > 0, \ j \geq -\langle \lambda, \alpha_i^\vee \rangle} \frac{1}{[j-1]!_i}F_i^{j-1} x_j
\end{align*}
for $\lambda \in P$ and $x \in M_\lambda$, and extend this by linearity. We call $\tilde{f}_i$ the {\it lowering Kashiwara operator}, and $\tilde{e}_i$ the {\it raising Kashiwara operator}.

We define a ring $A$ by
\begin{align*}
A := \mathbb{Q}[q]_{(q)} = \left\{ \left. \frac{f}{g} \ \right| \ f, g \in \mathbb{Q}[q], \ g(0) \not= 0 \right\}.
\end{align*}

\begin{defn}[{\cite[Definition 9.3]{Jantzen}}]
Let $M$ be a finite dimensional $U_q(\mathfrak{g})$-module. $A$-submodule $\mathcal{M}$ of $M$ is called an {\it admissible lattice} if $\mathcal{M}$ satisfies the following conditions.
\begin{enumerate}
\item $\mathcal{M}$ is finitely generated over $A$ and $\mathcal{M} \otimes_A k$ is isomorphic to $M$ as $k$-vector spaces by the multiplication map.
\item $\mathcal{M}$ has the weight space decomposition, that is, we have $\mathcal{M} = \bigoplus_{\lambda \in P} \mathcal{M}_\lambda$, where $\mathcal{M}_\lambda := \mathcal{M} \cap M_\lambda$ for $\lambda \in P$.
\item $\mathcal{M}$ is stable under Kashiwara operators, that is, for all $i \in I$, $\tilde{f}_i (\mathcal{M})$ and $\tilde{e}_i (\mathcal{M})$ are contained in $\mathcal{M}$.
\end{enumerate}
\end{defn}

Kashiwara operators $\tilde{f}_i$ and $\tilde{e}_i$, $i \in I$, on $M$ induce operators on $\mathcal{M}/q\mathcal{M}$. We also denote these operators by $\tilde{f}_i$ and $\tilde{e}_i$ for $i \in I$.

\begin{defn}[{\cite[Definition 9.4]{Jantzen}}]
Let $M$ be a finite dimensional $U_q(\mathfrak{g})$-module. A pair $(\mathcal{M}, \mathcal{B})$, where $\mathcal{M}$ is an admissible lattice of $M$ and $\mathcal{B}$ is a basis of the $\mathbb{Q}$-vector space $\mathcal{M}/q\mathcal{M}$, is called a {\it crystal basis} if $(\mathcal{M}, \mathcal{B})$ satisfies the following conditions.
\begin{enumerate}
\item $\mathcal{B}$ is compatible with the weight space decomposition of $\mathcal{M}$, that is, $\mathcal{B} = \bigsqcup_{\lambda \in P} \mathcal{B}_\lambda$ holds where $\mathcal{B}_\lambda := \mathcal{B} \cap (\mathcal{M}_\lambda/q\mathcal{M}_\lambda)$.
\item $\tilde{f}_i (\mathcal{B}) \subset \mathcal{B} \sqcup \{ 0 \}$ and $\tilde{e}_i (\mathcal{B}) \subset \mathcal{B} \sqcup \{ 0 \}$ for all $i \in I$.
\item For all $b, b^\prime \in \mathcal{B}$ and $i \in I$, we have $b^\prime = \tilde{f}_i (b) \Leftrightarrow b = \tilde{e}_i (b^\prime)$.
\end{enumerate}
\end{defn}

Now we construct the crystal basis of $L(\lambda)$ for $\lambda \in P^+$. Set
\begin{align*}
\mathcal{L}(\lambda) := \sum_{\substack{r \geq 0 \\ i_1, \ldots, i_r \in I}} A\tilde{f}_{i_1} \cdots \tilde{f}_{i_r}(v_\lambda).
\end{align*}
We set $b_\lambda := v_\lambda + q\mathcal{L}(\lambda) \in \mathcal{L}/q\mathcal{L}$. Then we define $\mathcal{B}(\lambda)$ by
\begin{align*}
\mathcal{B}(\lambda) := \left\{ \left. \tilde{f}_{i_1} \cdots \tilde{f}_{i_r} (b_\lambda) \ \right| \ r \geq 0 , \ i_1, \ldots, i_r \in I \right\} \setminus \{ 0 \}.
\end{align*}
Note that we can define operators $\tilde{f}_i$, $i \in I$, on $\mathcal{L}/q\mathcal{L}$ by $\tilde{f}_{i} (m + q\mathcal{L}(\lambda)) := \tilde{f}_i(m) + q\mathcal{L}(\lambda)$ for $m + q\mathcal{L}(\lambda) \in \mathcal{L}(\lambda)/q\mathcal{L}(\lambda)$.

Then we know the following theorem.
\begin{thm}[{\cite[Theorem 9.25]{Jantzen}}]
For all $\lambda \in P^+$, the pair $(\mathcal{L}(\lambda), \mathcal{B}(\lambda))$ is a crystal basis of $L(\lambda)$.
\end{thm}

Sometimes we call $\mathcal{B}(\lambda)$ the highest weight crystal for $\lambda \in P^+$.

\section{Basic Facts about Weyl Groups}\label{Weyl}
In this section, we recall some basic facts about Weyl groups. For details, see \cite{Humphreys}.

\subsection{The length function}
For $w \in W$, we can express $w$ as a product of simple reflections, that is, there exist $i_1, \ldots, i_l \in I$ such that $w = s_{i_1} \cdots s_{i_l}$. We call the product $s_{i_1} \cdots s_{i_l}$ an {\it expression} for $w$. If $w$ is written as $w = s_{i_1} \cdots s_{i_l}$ with $l$ minimal, then the product $s_{i_1} \cdots s_{i_l}$ is called a {\it reduced expression} for $w$, the sequence $(i_1, \ldots, i_l)$ is called a {\it reduced word} for $w$, and $l$ is called the {\it length} of $w$, denoted by $\ell(w)$. Note that a reduced expression for $w$ is not unique.

Also, for $w \in W$ we define $n(w) := \# (\Phi^+ \cap w^{-1}(-\Phi^+))$.

\begin{lem}[{\cite[Lemma 1.6]{Humphreys}}]\label{length1}
Let $k \in I$ and $w \in W$. Then the following holds.
\begin{enumerate}
\item $w\alpha_k \in \Phi^+ \Rightarrow n(ws_k) = n(w)+1$.
\item $w\alpha_k \in -\Phi^+ \Rightarrow n(ws_k) = n(w)-1$.
\item $w^{-1}\alpha_k \in \Phi^+ \Rightarrow n(s_kw) = n(w)+1$.
\item $w^{-1}\alpha_k \in -\Phi^+ \Rightarrow n(s_kw) = n(w)-1$.
\end{enumerate}
\end{lem}

\subsection{Deletion and exchange conditions}

First, we explain the deletion condition.
\begin{thm}[Deletion condition; {\cite[Theorem 1.7]{Humphreys}}]\label{del}
Let $w$ be an element of $W$. Let $w = s_{i_1} \cdots s_{i_l}$ be an arbitrary expression for $w$. If $n(w) < l$, then there exist indices $j, k \in \{1, \ldots, l\}$ with $j < k$ such that $w = s_{i_1} \cdots s_{i_{j-1}} s_{i_{j+1}} \cdots s_{i_{k-1}} s_{i_{k+1}} \cdots s_{i_l}$.
\end{thm}

As a corollary, we have the following.

\begin{cor}[{\cite[Section 1.6, 1.7]{Humphreys}}]\label{length2}
For $w \in W$, we have $\ell(w) = n(w)$.
\end{cor}
By Theorem \ref{del}, if an expression $w=s_{i_1} \cdots w_{i_l}$ for $w \in W$ is not a reduced expression, then we can omit two $s_i$'s.

Now we describe the exchange condition.

\begin{thm}[Exchange condition; {\cite[Section 1.7]{Humphreys}}]
Let $w$ be an element of $W$. Fix a reduced expression $w = s_{i_1} \cdots s_{i_l}$. If $\ell (ws_k) < \ell(w)$ for $k \in I$, then there exists an index $j \in \{1, \ldots l \}$ such that $ws_k = s_{i_1} \cdots s_{i_{j-1}} s_{i_{j+1}} \cdots s_{i_l}$.
\end{thm}

As a corollary, we obtain the following assertion which we use in the proof of our main results.

\begin{cor}\label{keycor}
Let $w \in W$. We set $l = \ell(w)$.

\noindent (1) If $w \alpha_k \in -\Phi^+$ for $k \in I$, then there exists a reduced expression $w = s_{j_1} \cdots s_{j_l}$ such that $j_l = k$.

\noindent (2) If $w^{-1} \alpha_k \in -\Phi^+$ for $k \in I$, then there exists a reduced expression $w = s_{j_1} \cdots s_{j_l}$ such that $j_1 = k$.
\end{cor}

\subsection{Parabolic subgroups and coset representatives}

For a subset $J$ of $I$, we define $W_J := \langle s_i \ | \ i \in J \rangle \subset W$.

\begin{defn}
A \textit{parabolic subgroup} of $W$ is a subgroup of $W$ of the form $W_J$ for some $J \subset I$.
\end{defn}

A parabolic subgroup of $W$ has minimal and maximal-length representatives.

\begin{lem}[{\cite[Section 1.10]{Humphreys}}]
(1) \ For a subset $J$ of $I$ and $w \in W$, the coset $wW_J$ has a unique element $\lfloor w \rfloor^J$ such that $\ell (\lfloor w \rfloor^J) \leq \ell(v)$ for all $v \in wW_J$. We call $\lfloor w \rfloor^J$ the \textit{minimal-length representative} for the coset $wW_J$.

\noindent (2) \ For a subset $J$ of $I$ and $w \in W$, the coset $wW_J$ has a unique element $\lceil w \rceil^J$ such that $\ell (\lceil w \rceil^J) \geq \ell(v)$ for all $v \in wW_J$. We call $\lceil w \rceil^J$ the \textit{maximal-length representative} for the coset $wW_J$.
\end{lem}

For $\lambda \in P^+$, the stabilizer of $\lambda$ is defined by $W_\lambda := \{ w \in W \ | \ w\lambda = \lambda \}$; $W_\lambda$ is a parabolic subgroup of $W$ since $W_\lambda = \langle s_i \ | \ i \in I, \ s_i\lambda = \lambda \rangle$. Hence, for each $w \in W$, the coset $wW_\lambda$ has minimal and maximal-length representatives. We denote by $\lfloor w \rfloor^\lambda$ the minimal-length representative for $wW_\lambda$, and by $\lceil w \rceil^\lambda$ the maximal-length representative for $wW_\lambda$.

\section{Crystals, Tensor Products and Demazure Crystals}\label{crystal}
In this section, we review the definitions and some of the properties of crystals, tensor products of crystals, and Demazure crystals.

\subsection{The definitions of crystals and tensor products of crystals}
First of all, we give the definition of crystals. Crystal bases of finite dimensional $U_q(\mathfrak{g})$-modules are examples of crystals.

\begin{defn}[{\cite[Definition 4.5.1, Remark 4.2.4]{Hong}}]
(1) A set $\mathcal{B}$ (which does not contain $0$), equipped with maps $\mathrm{wt} : \mathcal{B} \rightarrow P$, $\tilde{f}_i, \tilde{e}_i : \mathcal{B} \rightarrow \mathcal{B} \sqcup \{ 0 \}$, $i \in I$, $\varphi_i, \varepsilon_i : \mathcal{B} \rightarrow \mathbb{Z}\cup \{ -\infty \}$, $i \in I$, is called a {\it crystal} if it satisfies the following conditions:
\begin{align*}
\begin{array}{lll}
\varphi_i(b) = \varepsilon_i(b) + \langle \wt(b), \alpha_i^\vee \rangle & (b \in \mathcal{B}, i \in I);\\
\tilde{e}_i(b) \in \mathcal{B} \Rightarrow \wt(\tilde{e}_i(b)) = \wt(b) + \alpha_i &  (b \in \mathcal{B}, i \in I);\\
\tilde{f}_i(b) \in \mathcal{B} \Rightarrow \wt(\tilde{f}_i(b)) = \wt(b) - \alpha_i & ( b \in \mathcal{B}, i \in I);\\
\tilde{e}_i(b) \in \mathcal{B} \Rightarrow \varepsilon_i(\tilde{e}_i(b)) = \varepsilon_i(b)-1, \ \varphi_i(\tilde{e}_i(b)) = \varphi_i(b)+1 & (b \in \mathcal{B}, i \in I);\\
\tilde{f}_i(b) \in \mathcal{B} \Rightarrow \varepsilon_i(\tilde{f}_i(b)) = \varepsilon_i(b)+1, \ \varphi_i(\tilde{f}_i(b)) = \varphi_i(b)-1 & (b \in \mathcal{B}, i \in I);\\
\tilde{f}_i(b) = b^\prime \Leftrightarrow b = \tilde{e}_i(b^\prime) & (b, b^\prime \in \mathcal{B}, i \in I);\\
\varphi_i(b) = -\infty \Rightarrow \tilde{e}_i(b) = \tilde{f}_i(b) = 0 & (b \in \mathcal{B}, i \in I).
\end{array}
\end{align*}

\noindent (2) Let $\mathcal{B}$ be a crystal. The {\it crystal graph} of $\mathcal{B}$ is an $I$-colored directed graph whose vertices are the elements of $\mathcal{B}$ and whose $I$-colored edges are defined by: $b \xrightarrow{i} b^\prime \Leftrightarrow b^\prime = \tilde{f}_i (b)$ for $b, b^\prime \in \mathcal{B}$. 

\noindent (3) Let $\mathcal{B}$ be a crystal. A {\it connected component} of $\mathcal{B}$ is the set of vertices of a connected component of the crystal graph of $\mathcal{B}$. 

\noindent (4) Let $\mathcal{B}$ be a crystal. The {\it character} $\ch(\mathcal{B})$ is defined by
\begin{align*}
\ch(\mathcal{B}) := \sum_{\lambda \in P}(\# \mathcal{B}_\lambda)e^{\lambda},
\end{align*}
where $\mathcal{B}_\lambda := \{ b \in \mathcal{B} \ | \ \wt (b) = \lambda \}$ for $\lambda \in P$, and $e^\lambda$, $\lambda \in P$, are the formal basis elements of the group algebra $\mathbb{Z}[P]$ of $P$ with multiplication defined by $e^\lambda e^\mu := e^{\lambda + \mu}$ for $\lambda, \mu \in P$.
\end{defn}

\begin{exm}
Let $\lambda$ be a dominant integral weight. We define maps $\wt : \mathcal{B}(\lambda) \rightarrow P$ and $\varphi_i, \varepsilon_i : \mathcal{B}(\lambda) \rightarrow \mathbb{Z}$, $i \in I$, by
\begin{align*}
\wt (b) &:= \mu \ \text{ if } b \in \mathcal{B}(\lambda)_\mu \text{ for } \mu \in P, \\
\varphi_i (b) &:= \max \{ n \geq 0 \ | \ \tilde{f}_i^n (b) \not= 0 \}, \\
\varepsilon_i (b) &:= \max \{ n \geq 0 \ \left| \ \tilde{e}_i^n (b) \not= 0 \right. \}.
\end{align*}
Then it is easy to verify that $\mathcal{B}(\lambda)$ is a crystal equipped with maps $\mathrm{wt}$, $\{ \tilde{f}_i \}_{i \in I}$, $\{ \tilde{e}_{i} \}_{i \in I}$, $\{ \varphi_i \}_{i \in I}$, $\{ \varepsilon_i \}_{i \in I}$.

Also, we define maps $\tilde{e}_i^{\max}, \tilde{f}_i^{\max} : \mathcal{B}(\lambda) \to \mathcal{B}(\lambda)$ for $i \in I$ as follows. For $b \in \mathcal{B}(\lambda)$, we set $\tilde{e}_i^{\max}(b) := \tilde{e}_i^{\varepsilon_i(b)}(b)$ and $\tilde{f}_i^{\max}(b) := \tilde{f}_i^{\varphi_i(b)}(b)$. By definitions, we have $\tilde{e}_i^{\max}(b) \not= 0$, $\tilde{e} (\tilde{e}_i^{\max}(b)) = 0$, and $\tilde{f}_i^{\max}(b) \not= 0$, $\tilde{f}_i(\tilde{f}_i^{\max}(b)) = 0$ for $b \in \mathcal{B}(\lambda)$ and $i \in I$.
\end{exm}

We introduce morphisms between two crystals.

\begin{defn}[{\cite[Definition 4.5.5, Definition 4.5.6]{Hong}}]
Let $\mathcal{B}, \mathcal{B}^\prime$ be crystals.

\noindent (1) A {\it morphism} $\Psi : \mathcal{B} \rightarrow \mathcal{B}^\prime$ is a map $\Psi : \mathcal{B}\sqcup \{ 0 \} \rightarrow \mathcal{B}^\prime \sqcup \{ 0 \}$ satisfying:
\begin{enumerate}[{(}i{)}]
\item $\psi(0) = 0$;
\item if $b, b^\prime \in \mathcal{B}$ satisfy $b^\prime = \tilde{f}_i(b)$ for some $i \in I$ and $\Psi(b) \not= 0$, $\Psi(b^\prime) \not= 0$, then $\Psi(b^\prime) = \tilde{f}_i(\Psi(b))$ and $\Psi(b) = \tilde{e}_i(\Psi(b^\prime))$;
\item if $b \in \mathcal{B}$ satisfies $\Psi(b) \not= 0$, then $\wt(\Psi(b)) = \wt(b)$, $\varphi_i(\Psi(b)) = \varphi_i(b)$, and $\varepsilon_i(\Psi(b)) = \varepsilon_i(b)$ for all $i \in I$.
\end{enumerate}

\noindent (2) A morphism $\Psi : \mathcal{B} \rightarrow \mathcal{B}^\prime$ is {\it strict} if for all $b \in \mathcal{B}$ and $i \in I$, $\Psi(\tilde{e}_i(b)) = \tilde{e}_i(\Psi(b))$ and $\Psi(\tilde{f}_i(b)) = \tilde{f}_i(\Psi(b))$

\noindent (3) A strict morphism $\Psi : \mathcal{B} \rightarrow \mathcal{B}^\prime$ is called an {\it isomorphism} if $\Psi : \mathcal{B} \sqcup \{ 0 \} \rightarrow \mathcal{B}^\prime \sqcup \{ 0 \}$ is bijective.

\noindent (4) If there exists an isomorphism between $\mathcal{B}$ and $\mathcal{B}^\prime$, we say that $\mathcal{B}$ and $\mathcal{B}^\prime$ are {\it isomorphic}.
\end{defn}

When we take two crystals, we can form a tensor product of these.

\begin{defn}[{\cite[Definition 4.5.3]{Hong}}]
Let $\mathcal{B}$ and $\mathcal{B}^\prime$ be crystals. The {\it tensor product} of $\mathcal{B}$ and $\mathcal{B}^\prime$ is the set $\mathcal{B} \otimes \mathcal{B}^\prime := \mathcal{B} \times \mathcal{B}^\prime$ equipped with maps $\wt$, $\{ \tilde{f}_i \}_{i \in I}$, $\{ \tilde{e}_{i} \}_{i \in I}$, $\{ \varphi_i \}_{i \in I}$, $\{ \varepsilon_i \}_{i \in I}$ defined as follows: for $b \in \mathcal{B}$ and $b^\prime \in \mathcal{B}^\prime$,
\begin{align*}
\wt (b\otimes b^\prime) &= \wt(b) + \wt(b^\prime); \\
\varepsilon_i (b \otimes b^\prime) &= \max \{ \varepsilon_i(b), \varepsilon_i(b^\prime)-\langle \wt(b), \alpha_i^\vee \rangle \}; \\
\varphi_i (b \otimes b^\prime) &= \max \{ \varphi_i(b^\prime), \varphi_i(b) + \langle \wt(b^\prime), \alpha_i^\vee \rangle \}; \\
\tilde{e}_i(b \otimes b^\prime) &= \begin{cases} \tilde{e}_i(b)\otimes b^\prime & (\varphi_i(b) \geq \varepsilon_i(b^\prime)), \\ b \otimes \tilde{e}_i(b^\prime) & (\varphi_i(b) < \varepsilon_i(b^\prime)); \end{cases} \\
\tilde{f}_i(b \otimes b^\prime) &= \begin{cases} \tilde{f}_i(b)\otimes b^\prime & (\varphi_i(b) > \varepsilon_i(b^\prime)), \\ b \otimes \tilde{f}_i(b^\prime) & (\varphi_i(b) \leq \varepsilon_i(b^\prime)). \end{cases}
\end{align*}
Note that for $b \in \mathcal{B}$ and $b^\prime \in \mathcal{B}^\prime$, we denote the element $(b, b^\prime) \in \mathcal{B} \otimes \mathcal{B}^\prime$ by $b \otimes b^\prime$.
\end{defn}

We can easily check that a tensor products of crystals is also a crystal.

Let $\mathcal{B}$ and $\mathcal{B}^\prime$ be crystals. In this paper, we set $\mathcal{S} \otimes \mathcal{S}^\prime := \{ b \otimes b^\prime \ | \ b \in \mathcal{S}, \ b^\prime \in \mathcal{S}^\prime \}$ for subsets $\mathcal{S} \subset \mathcal{B}$ and $\mathcal{S}^\prime \subset \mathcal{B}^\prime$.

For $\lambda, \mu \in P^+$, $i \in I$, and $b \otimes b^\prime \in \mathcal{B}(\lambda) \otimes \mathcal{B}(\mu)$, we set $\tilde{e}_i^{\max}(b \otimes b^\prime) := \tilde{e}_i^{\varepsilon_i(b \otimes b^\prime)}(b \otimes b^\prime)$ and $\tilde{f}_i^{\max}(b \otimes b^\prime) := \tilde{f}_i^{\varphi_i(b \otimes b^\prime)}(b \otimes b^\prime)$. We can verify that $\tilde{e}_i^{\max}(b \otimes b^\prime) \not= 0$, $\tilde{e}_i( \tilde{e}_i^{\max}(b \otimes b^\prime) ) = 0$, and $\tilde{f}_i^{\max}(b \otimes b^\prime) \not= 0$, $\tilde{f}_i( \tilde{f}_i^{\max}(b \otimes b^\prime) ) = 0$.

\subsection{Demazure crystals}
We define Demazure crystals $\mathcal{B}_w(\lambda)$ for $w \in W$ and $\lambda \in P^+$.

\begin{defn}
Let $w \in W$ and $\lambda \in P^+$. We fix a reduced expression $w = s_{i_1} \cdots s_{i_l}$, where $l = \ell (w)$. The {\it Demazure crystal} $\mathcal{B}_w(\lambda)$ is defined to be
\begin{align*}
\mathcal{B}_w(\lambda) := \left\{ \left. \tilde{f}_{i_1}^{a_1} \cdots \tilde{f}_{i_l}^{a_l} (b_\lambda) \ \right| \ a_1, \ldots, a_l \geq 0 \right\} \setminus \{ 0 \}.
\end{align*}
\end{defn}

\begin{rem}
(1) The Demazure crystal $\mathcal{B}_w(\lambda)$ does not depend on the choice of a reduced expression for $w$. Moreover, for all $i \in I$, we have $\tilde{e}_i(\mathcal{B}_w(\lambda)) \subset \mathcal{B}_w(\lambda) \sqcup \{ 0 \}$. For details, see \cite[Proposition 3.2.3]{Kashiwara}. 

\noindent (2) In general, $\mathcal{B}_w(\lambda)$ is not stable under lowering Kashiwara operators. Hence a Demazure crystal is not a crystal.
\end{rem}

A Demazure crystal has a good parametrization.

\begin{defn}[{\cite[Section 1]{Littelmann}}]
Let $w \in W$ and $\lambda \in P^+$. We fix a reduced expression $w = s_{i_1} \cdots s_{i_l}$, where $l = \ell (w)$. For $b \in \mathcal{B}_w(\lambda)$, we define the $l$-tuple of nonnegative integers $\Omega(b) := (a_1, \ldots, a_l)$ by
\begin{align*}
a_1 &:= \max \{ a \geq 0 \ | \ \tilde{e}_{i_1}^a (b) \not= 0 \}, \\
a_2 &:= \max \{ a \geq 0 \ | \ \tilde{e}_{i_2}^a \tilde{e}_{i_1}^{a_1} (b) \not= 0 \}, \\
a_3 &:= \max \{ a \geq 0 \ | \ \tilde{e}_{i_3}^{a} \tilde{e}_{i_2}^{a_2} \tilde{e}_{i_1}^{a_1} (b) \not= 0 \}, \\
&\vdots& \\
a_l &:= \max \{ a \geq 0 \ | \ \tilde{e}_{i_l}^{a} \tilde{e}_{i_{l-1}}^{a_{l-1}} \cdots \tilde{e}_{i_1}^{a_1} (b) \not= 0 \}.
\end{align*}
Then one has $b = \tilde{f}_{i_1}^{a_1} \cdots \tilde{f}_{i_l}^{a_l} (b_\lambda)$, and for all $k \in \{1, \ldots, l-1 \}$ we have $\tilde{e}_{i_k} \tilde{f}_{i_{k+1}}^{a_{k+1}} \cdots \tilde{f}_{i_l}^{a_l} (b_\lambda) = 0$. We call $\Omega(b)$, or the equation $b = \tilde{f}_{i_1}^{a_1} \cdots \tilde{f}_{i_l}^{a_l} (b_\lambda)$ the {\it string parametrization} of $b$.
\end{defn}

\begin{rem}
It is known that $\Omega : \mathcal{B}_w(\lambda) \rightarrow \mathbb{Z}^l$ is injective.
\end{rem}

We know the following.

\begin{lem}[{\cite[Proposition 3.2.3]{Kashiwara}}]\label{f(Dem)}
Let $\lambda \in P^+$, $w \in W$ and $k \in I$.

\noindent (1) If $\ell(s_k w) > \ell(w)$, then we have
\begin{align*}
\bigcup_{n\geq0} \tilde{f}_k^n (\mathcal{B}_w(\lambda)) \setminus \{ 0 \} = \mathcal{B}_{s_k w}(\lambda).
\end{align*}

\noindent (2) If $\ell(s_k w) < \ell(w)$, then we have
\begin{align*}
\bigcup_{n\geq0} \tilde{f}_k^n (\mathcal{B}_w(\lambda)) \setminus \{ 0 \} = \mathcal{B}_{w}(\lambda).
\end{align*}

In particular, $\tilde{f}_k(\mathcal{B}_w(\lambda)) \subset \mathcal{B}_w(\lambda) \sqcup \{ 0 \}$ if $\ell(s_k w) < \ell(w)$.
\end{lem}

For a crystal $\mathcal{B}$ and $i \in I$, an \textit{$i$-string} is a subset of $\mathcal{B}$ of the form:
\begin{align*}
S = \{ \tilde{e}_i^n (b) \ | \ n \geq 0 \} \cup \{ \tilde{f}_i^n (b) \ | \ n \geq 0 \} \setminus \{ 0 \}
\end{align*}
for some $b \in \mathcal{B}$. For an $i$-string $S$, the {\it highest weight element} of $S$ is the element $b \in S$ such that $\tilde{e}_i(b) = 0$.

A Demazure crystal has the following property, called the \textit{string property}.

\begin{lem}[{\cite[Proposition 3.3.5]{Kashiwara}}]\label{string}
Let $\lambda \in P^+$, $w \in W$, and $i \in I$. For an $i$-string $S \subset \mathcal{B}(\lambda)$ with the highest weight element $b$, the set $S \cap \mathcal{B}_w(\lambda)$ is identical to either $\emptyset$, $\{ b \}$, or $S$.
\end{lem}

\section{Lakshmibai-Seshadri paths}\label{LS}

\subsection{Definition of Lakshmibai-Seshadri paths}
In this subsection, we recall the definition of Lakshmibai-Seshadri paths, which provide a realization of highest weight crystals. For details, see \cite{Littelmann_LR} and \cite{Littelmann_path}.

Let $P_{\mathbb{R}} := P \otimes_{\mathbb{Z}} \mathbb{R}$ be the vector space over $\mathbb{R}$ generated by $P$. A piecewise-linear continuous map $\pi : [0, 1] \rightarrow P_\mathbb{R}$ with $\pi(0) = 0$ is called a {\it path}. Now we define Lakshmibai-Seshadri paths. Let $\lambda \in P^+$. 

\begin{defn}
For $\mu, \nu \in W\lambda$, we write $\mu \geq \nu$ if there exist $r \geq 0$, $\mu_0, \mu_1, \ldots, \mu_r \in W\lambda$ with $\mu_0 = \mu, \ \mu_r = \nu$, and positive roots $\beta_1, \ldots, \beta_r \in \Phi^+$ such that for $k = 1, \ldots, r$,
\begin{align*}
\mu_{k} = s_{\beta_k}(\mu_{k-1}), \ \langle \mu_{k-1} , \beta_k^\vee \rangle < 0.
\end{align*}
When $\mu \geq \nu$ holds for $\mu, \nu \in W\lambda$, the number $\dist (\mu, \nu)$ is defined to be the maximal length $r$ of sequences satisfying the above condition.
\end{defn}

\begin{defn}
Let $\sigma$ be a rational number with $0 < \sigma < 1$. For $\mu, \nu \in W\lambda$ with $\mu \geq \nu$, a {\it $\sigma$-chain} for $(\mu, \nu)$ is a sequence $\mu = \mu_0 > \mu_1 > \cdots > \mu_r = \nu$ of elements in $W\lambda$ with $r \geq 0$, and positive roots $\beta_1, \ldots, \beta_r \in \Phi^+$ such that either $r = 0$, or $r \geq 1$ and for each $k \in \{1, \ldots, r \}$, we have $\mu_k = s_{\beta_{k}}(\mu_{k-1})$, $\dist(\mu_{k-1}, \mu_k) = 1$, and $\sigma \langle \mu_{k-1}, \beta_k^\vee \rangle \in \mathbb{Z}$.
\end{defn}

\begin{defn}
(1) A {\it Lakshmibai-Seshadri chain} of shape $\lambda$ is a pair $(\underline{\nu}; \underline{a})$ where $\underline{\nu} : \nu_1 > \cdots > \nu_r$ is a sequence of elements in $W\lambda$, and $\underline{a} : 0 = a_0 < a_1 < \cdots < a_r = 1$ is a sequence of rational numbers such that there is an $a_k$-chain for $(\nu_k, \nu_{k+1})$ for each $k \in \{ 1, \ldots, r-1 \}$. We call Lakshmibai-Seshadri chains LS-chains for short.

\noindent (2) Let $(\underline{\nu}; \underline{a}) = (\nu_1 > \cdots > \nu_r; \ 0 = a_0 < a_1 < \cdots < a_r = 1)$ be an LS-chain of shape $\lambda$. A {\it Lakshmibai-Seshadri path} corresponding to $(\underline{\nu}; \underline{a})$ is a path $\pi$ defined by
\begin{align*}
\pi (t) := \sum_{i=1}^{k-1} (a_i - a_{i-1})\nu_i + (t-a_{k-1})\nu_k
\end{align*}
if $t \in [a_{k-1}, a_k]$ for $k \in \{1, \ldots, r-1 \}$.
We call Lakshmibai-Seshadri paths LS-paths for short. When $\pi$ is an LS-path corresponding to an LS-chain $(\underline{\nu}; \underline{a})$, we often write $\pi =  (\underline{\nu}; \underline{a})$.

\noindent (3) We denote the set of LS-paths of shape $\lambda$ by $\mathbb{B}(\lambda)$.
\end{defn}

\subsection{Crystal structure on the set of Lakshmibai-Seshadri paths}

We define a crystal structure on $\mathbb{B}(\lambda)$.

\begin{defn}
For a path $\pi$, we set $\wt (\pi) := \pi(1)$. We call $\wt(\pi)$ the {\it weight} of $\pi$.
\end{defn}

\begin{rem}
By \cite[Lemma 4.5 a)]{Littelmann_path}, we have $\pi(1) \in P$ for each $\pi \in \mathbb{B}(\lambda)$; hence $\wt$ defines a map from $\mathbb{B}(\lambda)$ to $P$.
\end{rem}

Fix a path $\pi : [0, 1] \rightarrow P_\mathbb{R}$ and $i \in I$. We set
\begin{align*}
h_i^\pi : [0, 1] \rightarrow \mathbb{R}, \ h_i^\pi(t) := \langle \pi(t), \alpha_i^\vee \rangle;
\end{align*}
this is called the {\it height function} for $\pi$.

Assume that all of minimums of the function $h_i^\pi(t)$ are integers and that $\pi(1) \in P$. If $\pi$ is an LS-path, these assumptions are satisfied. Set $m_i^\pi := \min \{ h_i^\pi(t) \ | \ t \in [0, 1] \}$. Note that $m_i^\pi \in \mathbb{Z}$. Moreover, $m_i^\pi \leq 0$ since $\pi(0) = 0$. Also, since $\pi(1) \in P$, $h_i^\pi(1) = \langle \pi(1), \alpha_i^\vee \rangle \in \mathbb{Z}$ holds. Now we can define root operators $\tilde{e}_i$ and $\tilde{f}_i$.

First, we introduce the {\it raising root operator} $\tilde{e}_i$. If $m_i^\pi = 0$, then we set $\tilde{e}_i (\pi) := 0$. If $m_i^\pi < 0$, then we set $t_1 := \min \{ t \in [0, 1] \ | \ h_i^\pi(t) = m_i^\pi \}$, $t_0 := \max \{ t \in [0, t_1] \ | \ h_i^\pi(t) = m_i^\pi +1 \}$. Note that the set $\{ t \in [0, t_1] \ | \ h_i^\pi(t) = m_i^\pi +1 \}$ is not empty because $h_i^\pi(0) = 0$, $m_i^\pi \in \{ -1, -2, \ldots \}$, and $h_i^\pi$ is continuous. Hence $0 \leq t_0 < t_1 \leq 1$ holds. We define $\tilde{e}_i(\pi)$ by

\begin{align*}
(\tilde{e}_i (\pi))(t) := \begin{cases} \pi(t) & (0 \leq t \leq t_0), \\
s_i(\pi(t) - \pi(t_0)) + \pi(t_0) & (t_0 \leq t \leq t_1), \\
\pi(t) + \alpha_i & (t_1 \leq t \leq 1). \end{cases}
\end{align*}

Next we define the {\it lowering root operator} $\tilde{f}_i$. If $m_i^\pi = h_i^\pi(1)$, then we set $\tilde{f}_i (\pi) := 0$. If $m_i^\pi < h_i^\pi(1)$, then we set $t_0 := \max \{ t \in [0, 1] \ | \ h_i^\pi(t) = m_i^\pi \}$, $t_1 := \min \{ t \in [t_0, 1] \ | \ h_i^\pi(t) = m_i^\pi +1 \}$. Note that the set $\{ t \in [t_0, 1] \ | \ h_i^\pi(t) = m_i^\pi +1 \}$ is not empty because $h_i^\pi(1) \in \mathbb{Z}$, $m_i^\pi \in \{ \ldots, h_i^\pi(1) - 2, h_i^\pi(1) - 1\}$, and $h_i^\pi$ is continuous. Hence $0 \leq t_0 < t_1 \leq 1$ holds. We define $\tilde{f}_i(\pi)$ by

\begin{align*}
(\tilde{f}_i (\pi))(t) := \begin{cases} \pi(t) & (0 \leq t \leq t_0), \\
s_i(\pi(t) - \pi(t_0)) + \pi(t_0) & (t_0 \leq t \leq t_1), \\
\pi(t) - \alpha_i & (t_1 \leq t \leq 1). \end{cases}
\end{align*}

\begin{lem}[{\cite[Section 2, Section 4]{Littelmann_path}}]
For $i \in I$, we have $\tilde{e}_i ( \mathbb{B}(\lambda) ) \subset \mathbb{B}(\lambda) \sqcup \{ 0 \}$ and $\tilde{f}_i ( \mathbb{B}(\lambda) ) \subset \mathbb{B}(\lambda) \sqcup \{ 0 \}$.
\end{lem}

Note that all of minimums of the functions $h_i^{\tilde{e}_i(\pi)}(t)$ and $h_i^{\tilde{f}_i(\pi)}(t)$ are integers, and that $(\tilde{e}_i(\pi))(1)$ and $(\tilde{f}_i(\pi))(1)$ are contained in $P$.

Finally, for $i \in I$ and a path $\pi$ which satisfies that all of minimums of the function $h_i^\pi(t)$ are integers and that $\pi(1) \in P$, we set $\varphi_i(\pi) := \max \{ n \geq 0 \ | \ \tilde{f}_i^n(\pi) \not= 0 \}$ and $\varepsilon_i(\pi) := \max \{ n \geq 0 \ | \ \tilde{e}_i^n(\pi) \not= 0 \}$. By using these, we can define maps $\varphi_i, \varepsilon_i : \mathbb{B}(\lambda) \rightarrow \mathbb{Z}$. 
\begin{thm}[{\cite[Section 2, Section 4]{Littelmann_path}}]
The set $\mathbb{B}(\lambda)$, equipped with maps $\wt$, $\{ \tilde{f}_i \}_{i \in I}$, $\{ \tilde{e}_i \}_{i \in I}$, $\{ \varphi_i \}_{i \in I}$, $\{ \varepsilon_i \}_{i \in I}$, is a crystal.
\end{thm}

Moreover, we know the following.

\begin{thm}[{\cite[Theorem 4.1]{Kashiwara2}, \cite[Corollary 6.4.27]{Joseph}}]
The crystal $\mathbb{B}(\lambda)$ is isomorphic to $\mathcal{B}(\lambda)$ as crystals.
\end{thm}

From now on, we identify $\mathbb{B}(\lambda)$ with $\mathcal{B}(\lambda)$ for $\lambda \in P^+$. For $\lambda \in P^+$, we define the path $\pi^\lambda$ by $\pi^\lambda (t) := t\lambda$ for $t \in [0, 1]$. Then, for $\lambda \in P^+$, we have $\pi^\lambda \in \mathbb{B}(\lambda)$; we identify $\pi^\lambda \in \mathbb{B}(\lambda)$ with $b_\lambda \in \mathcal{B}(\lambda)$ for $\lambda \in P^+$.

For paths $\pi, \pi^\prime$, we define the {\it concatenation} $\pi \ast \pi^\prime$ of $\pi$, $\pi^\prime$ by
\begin{align*}
\pi \ast \pi^\prime (t) := \begin{cases} \pi (2t) & (0 \leq t \leq 1/2), \\
\pi(1) + \pi^\prime(2t-1) & (1/2 \leq t \leq 1). \end{cases}
\end{align*}
If all of minimums of the functions $h_i^{\pi}(t)$ and $h_i^{\pi^\prime}(t)$ are integers, and $\pi(1)$ and $\pi^\prime(1)$ are contained in $P$, then the path $\pi \ast \pi^\prime$ also satisfies that all of minimums of the function $h_i^{\pi \ast \pi^\prime}(t)$ are integers and that $(\pi \ast \pi^\prime)(1) \in P$. Hence we can use the definition of $\tilde{e}_i$, $\tilde{f}_i$, $\varphi_i$, $\varepsilon_i$, $i \in I$, above to define a crystal structure on $\mathcal{B}(\lambda) \ast \mathcal{B}(\mu) := \{ \pi \ast \pi^\prime \ | \ \pi \in \mathcal{B}(\lambda), \ \pi^\prime \in \mathcal{B}(\mu) \}$ for $\lambda, \mu \in P^+$.

\begin{thm}[{\cite[Section 2]{Littelmann_path}}]
For $\lambda, \mu \in P^+$, the set $\mathcal{B}(\lambda) \ast \mathcal{B}(\mu)$, equipped with maps $\wt$, $\{ \tilde{f}_i \}_{i \in I}$, $\{ \tilde{e}_i \}_{i \in I}$, $\{ \varphi_i \}_{i \in I}$, $\{ \varepsilon_i \}_{i \in I}$ defined as above, is a crystal. Moreover, we have $\mathcal{B}(\lambda) \ast \mathcal{B}(\mu) = \mathcal{B}(\lambda) \otimes \mathcal{B}(\mu)$ by identifying $\pi_1 \ast \pi_2 \in \mathcal{B}(\lambda) \ast \mathcal{B}(\mu)$ with $\pi_1 \otimes \pi_2 \in \mathcal{B}(\lambda) \otimes \mathcal{B}(\mu)$.
\end{thm}

Now, we describe the Littlewood-Richardson rule for crystal bases in terms of LS-paths, called the {\it generalized Littlewood-Richardson rule}.

\begin{defn}
Let $\lambda \in P^+$. A path $\pi$ is {\it $\lambda$-dominant} if $\langle \lambda + \pi (t), \alpha_i^\vee \rangle \geq 0$ for all $i \in I$ and $t \in [0, 1]$.
\end{defn}

For $\lambda, \mu \in P^+$, we denote the set of $\lambda$-dominant LS-paths of shape $\mu$ by $\mathcal{B}(\mu)^\lambda$. Also, for $\lambda, \mu \in P^+$ and $\pi \in \mathcal{B}(\mu)^\lambda$, the set $C(\pi, \lambda, \mu)$ is defined to be the connected component of $\mathcal{B}(\lambda) \otimes \mathcal{B}(\mu)$ containing $\pi^\lambda \otimes \pi$; we obtain the crystal structure on $C(\pi, \lambda, \mu)$ by restricting that of $\mathcal{B}(\lambda) \otimes \mathcal{B}(\mu)$.

\begin{thm}[{\cite[Section 10]{Littelmann_path}}]\label{LRLSver}
Let $\lambda, \mu$ be dominant integral weights. It holds that
\begin{align*}
\mathcal{B}(\lambda) \otimes \mathcal{B}(\mu) = \bigsqcup_{\pi \in \mathcal{B}(\mu)^\lambda} C(\pi, \lambda, \mu).
\end{align*}
Moreover, the connected component $C(\pi, \lambda, \mu)$ is isomorphic to $\mathcal{B}(\lambda + \wt(\pi))$ as crystals.
\end{thm}

At the end of this subsection, we describe Demazure crystals in terms of LS-paths. For details, see \cite{Kashiwara} and \cite{Littelmann_LR}.

\begin{defn}
Let $\pi$ be an LS-path. If $\pi$ corresponds to an LS-chain $(\nu_1 > \cdots > \nu_r;0 = a_0 < \cdots < a_r = 1)$, then we call $\nu_1$ the {\it initial direction} of $\pi$, and set $\iota (\pi) := \nu_1$.
\end{defn}

\begin{thm}
Let $\lambda \in P^+$ and $w \in W$. Then one has $\mathcal{B}_w(\lambda) = \{ \pi \in \mathcal{B}(\lambda) \ | \ \iota (\pi) \leq w \lambda \}$.
\end{thm}

\subsection{Tensor product of a highest weight element and a Demazure crystal}

In \cite{Magyar}, it is proved that a tensor product of a highest weight element and a Demazure crystal decomposes into a disjoint union of Demazure crystals. 

First, we introduce some notation which are used through this paper. For $\lambda, \mu \in P^+$, we set $\mathcal{B}_w(\mu)^\lambda := \mathcal{B}(\mu)^\lambda \cap \mathcal{B}_w(\mu)$.

\begin{defn}
Let $\lambda, \mu$ be dominant integral weights, and $v, w$ elements of the Weyl group $W$. For $\pi \in \mathcal{B}_w(\mu)^\lambda$, we denote by $C(\pi, v, w, \lambda, \mu)$ the connected component of $\mathcal{B}_v(\lambda) \otimes \mathcal{B}_w(\mu)$ containing $\pi^\lambda \otimes \pi$.
\end{defn}

We omit $w, \lambda$ and $\mu$ from $C(\pi, v, w, \lambda, \mu)$, that is, we simply write $C(\pi, v)$ for $C(\pi, v, w, \lambda, \mu)$.

Note that it follows that
\begin{align*}
\mathcal{B}_v(\lambda) \otimes \mathcal{B}_w(\mu) = \bigsqcup_{\pi \in \mathcal{B}_w(\mu)^\lambda} C(\pi, v)
\end{align*}
by Theorem \ref{LRLSver}.

Let $\lambda, \mu$ be dominant integral weights, and $w \in W$. For $\pi \in \mathcal{B}_w(\mu)^\lambda$, we define $w(\pi) \in W$ as follows.

As the first step, we take a decomposition of [0,1]. First, let $W(t) := \{ w \in W \ | \ w(\lambda + \pi(t)) = \lambda + \pi(t) \}$ be a stabilizer of $\lambda + \pi(t) \in P_{\mathbb{R}}$ for $t \in [0, 1]$. Then there exist intervals $I_1, \ldots, I_q$ such that $[0, 1] = I_1 \sqcup \cdots \sqcup I_q$, and such that for all $i \in \{1, \ldots, q \}$ and $t, t^\prime \in I_i$ one has $W(t) = W(t^\prime)$, and $t < t^\prime$ if $t \in I_i$, $t^\prime \in I_j$ for each $i, j \in \{1, \ldots, q \}$ with $i < j$. We take such intervals so that $q$ is minimal.

Now we define $w(\pi)$. We denote the Bruhat order on $W$ by $\leq$. First, we set $w_{q+1} := e$, where $e$ is the identity element of $W$. When $w_j$ is defined for $j \in \{3, 4, \ldots, q+1\}$, let $w_{j-1} := \max (W(I_{j-1})w_j)$. Assume that $w_2, \ldots, w_{q+1}$ are defined. Then we set $u_1 := \max \{ u \in W(I_1) \ | \ u \tau_1 \leq w W_\mu \}$ and $w_1 := \max\{ uw_2 \ | \ u \leq u_1 \}$. Here, $\tau_1$ is defined by $\tau_1 := \min \{ w \in W \ | \ w \mu = \iota (\pi) \}$. Recall that $W_\mu = \{ w \in W \ | \ w\mu = \mu \}$ is the stabilizer of $\mu$. Finally, we set $w(\pi) := w_1$. We must discuss the existence of $u_1$. Also, we have to check that $w(\pi)$ is well-defined. For details, see \cite{Magyar}.

We can state the decomposition theorem for a tensor product of a highest weight element and a Demazure crystal. Recall that $\mathcal{B}_e(\lambda) = \{ \pi^\lambda \}$.

\begin{thm}[{\cite[Proposition 12]{Magyar}}]\label{1tensDem}
Let $\lambda$ and $\mu$ be dominant integral weights, and $w$ an element of $W$. For $\pi \in \mathcal{B}_w(\mu)^\lambda$, the connected component $C(\pi, e)$ is isomorphic to $\mathcal{B}_{w(\pi)}(\lambda + \wt(\pi))$. Hence one has the following isomorphism:
\begin{align*}
\mathcal{B}_e(\lambda) \otimes \mathcal{B}_w(\mu) \simeq \bigsqcup_{\pi \in \mathcal{B}_w(\mu)^\lambda} \mathcal{B}_{w(\pi)}(\lambda + \wt(\pi)).
\end{align*}
\end{thm}

The first main result of this paper is a generalization of this theorem, which is given in the next section.

\section{Tensor Products of Demazure Crystals}\label{mainsec}

In this section, we discuss the condition for every connected component of a tensor product of Demazure crystals to be isomorphic to a Demazure crystal.

Before the discussion, we introduce some notation. For $w \in W$, we set $D_L(w) := \{ i \in I \ | \ \ell(s_i w)< \ell(w) \}$, the {\it left descent set}. Then we define $W_w := W_{D_L(w)} = \langle s_i \ | \ i \in D_L(w) \rangle$, which is a parabolic subgroup of $W$. 

Next, we introduce notation for specific crystals. A {\it word} is a sequence $\text{\boldmath $i$} = (i_1 , \ldots, i_l)$ with $i_j \in I$ for $j \in \{ 1, \ldots, l \}$.

\begin{defn}
Let $\lambda, \mu \in P^+$, and $w \in W$.
\begin{enumerate}[{(}1{)}]
\item For a word $\text{\boldmath $i$} = (i_1, \ldots, i_l)$, we set
\begin{align*}
\mathcal{B}_{\text{\boldmath $i$}, w, \lambda, \mu} := \bigcup_{a_1, \ldots, a_l \geq 0} \tilde{f}_{i_1}^{a_1} \cdots \tilde{f}_{i_l}^{a_l} (\mathcal{B}_e(\lambda) \otimes \mathcal{B}_w(\mu) ) \setminus \{ 0 \}.
\end{align*}

\item For a word $\text{\boldmath $i$}$ and a path $\pi \in \mathcal{B}_w(\mu)^\lambda$, we denote by $D(\pi, \text{\boldmath $i$}, w, \lambda, \mu)$ the connected component of $\mathcal{B}_{\text{\boldmath $i$}, w, \lambda, \mu}$ containing $\pi^\lambda \otimes \pi$.
\end{enumerate}
\end{defn}

We omit $w, \lambda$ and $\mu$ from $D(\pi, \text{\boldmath $i$}, w, \lambda, \mu)$ as for $C(\pi, v)$, namely, we write $D(\pi, \text{\boldmath $i$})$ instead of $D(\pi, \text{\boldmath $i$}, w, \lambda, \mu)$.

\begin{rem}\label{rem_for_genDemcrys} Let $\lambda, \mu \in P^+$, and $w \in W$. Take a word {\boldmath $i$}.
\begin{enumerate}[{(}1{)}]
\item We can verify by direct computation that the set $\mathcal{B}_{\text{\boldmath $i$}, w, \lambda, \mu}$ is identical to a tensor product of some highest weight element and some generalized Demazure crystal. For details about generalized Demazure crystals, see \cite{Magyar}.

\item The connected component $D(\pi, \text{\boldmath $i$})$ is isomorphic to some Demazure crystal because each generalized Demazure crystal is a disjoint union of Demazure crystals by Theorem 2 of \cite{Magyar}.

\item If {\boldmath $i$} is a reduced word for $v \in W$, we have $C(\pi, v) \subset D(\pi, \text{\boldmath $i$})$ for $\pi \in \mathcal{B}_w(\mu)^\lambda$.
\end{enumerate}
\end{rem}

Now we consider a sufficient condition. We prove the following theorem.

\begin{thm}\label{mainthm_1}
Let $\lambda$ and $\mu$ be dominant integral weights, and $v$, $w$ elements of $W$. Fix a reduced expression $\lfloor v \rfloor^\lambda = s_{i_1} \cdots s_{i_l}$. If $\lfloor v \rfloor^\lambda \in W_{\lceil w \rceil^\mu}$, then we have $\mathcal{B}_v(\lambda) \otimes \mathcal{B}_w(\mu) = \mathcal{B}_{(i_1, \ldots, i_l), w, \lambda, \mu}$.
\end{thm}

By Remark \ref{rem_for_genDemcrys}, we can conclude that $\mathcal{B}_v(\lambda) \otimes \mathcal{B}_w(\mu)$ decomposes into a disjoint union of Demazure crystals if $\lfloor v \rfloor^\lambda \in W_{\lceil w \rceil^\mu}$. More precisely, we can give an explicit formula of such a decomposition in this case. To describe it, we define $u(\pi, v) \in W$ for $\lambda, \mu \in P^+$, $v, w \in W$ with $\lfloor v \rfloor^\lambda \in W_{\lceil w \rceil^\mu}$, and $\pi \in \mathcal{B}_w(\mu)^\lambda$. First, we fix a reduced expression $\lfloor v \rfloor^\lambda = s_{i_1} \cdots s_{i_l}$ for $v$. Then, we set
\begin{align*}
u_l := \begin{cases} s_{i_l} w(\pi) & \text{ if } \ell(s_{i_l} w(\pi)) > \ell(w(\pi)), \\ w(\pi) & \text{ if } \ell(s_{i_l} w(\pi)) < \ell(w(\pi)). \end{cases}
\end{align*} 
When $u_j \in W$ is defined for some $j \in \{2, \ldots, l\}$, then we define $u_{j-1} \in W$ by
\begin{align*}
u_{j-1} := \begin{cases} s_{i_{j-1}} u_j  & \text{ if } \ell(s_{i_{j-1}} u_j) > \ell(u_j), \\ u_j & \text{ if } \ell(s_{i_{j-1}} u_j) < \ell(u_j). \end{cases}
\end{align*}
Finally, we define $u(\pi, v) := u_1$. Note that $u(\pi, v)$ depends on the choice of a reduced word for $\lfloor v \rfloor^\lambda$.

We can write an explicit decomposition formula in terms of $u(\pi, v)$.

\begin{cor}\label{maincor}
Let $\lambda$ and $\mu$ be dominant integral weights, and $v, w$ elements of $W$. If $\lfloor v \rfloor^\lambda \in W_{\lceil w \rceil^\mu}$, then we have $C(\pi, v) \simeq \mathcal{B}_{u(\pi, v)}(\lambda + \wt(\pi))$ for all $\pi \in \mathcal{B}_w(\mu)^\lambda$. Hence it follows that
\begin{align*}
\mathcal{B}_v(\lambda) \otimes \mathcal{B}_w(\mu) \simeq \bigsqcup_{\pi \in \mathcal{B}_w(\mu)^\lambda} \mathcal{B}_{u(\pi, v)}(\lambda + \wt(\pi)).
\end{align*}
\end{cor}

Note that if $w$ is the longest element $w_{\circ}$ of $W$, then we have $\mathcal{B}_{w_\circ}(\mu) = \mathcal{B}(\mu)$ and $W_{\lceil w_\circ \rceil^\mu} = W$. Hence, for all $\lambda, \mu \in P^+$ and $v \in W$, $\mathcal{B}_v(\lambda) \otimes \mathcal{B}(\mu)$ is a disjoint union of Demazure crystals.

Next we consider a necessary condition. In fact, the following assertion holds.

\begin{thm}\label{mainthm_4}
Let $\lambda, \mu \in P^+$ and $v, w \in W$. If $\lfloor v \rfloor^\lambda \not\in W_{\lceil w \rceil^\mu}$, then there exists a connected component of $\mathcal{B}_v(\lambda) \otimes \mathcal{B}_w(\mu)$ that is not isomorphic to any Demazure crystal.
\end{thm}

Therefore, by combining Corollary \ref{maincor} and Theorem \ref{mainthm_4}, we obtain the following theorem, which is one of the main results of this paper.

\begin{thm}\label{maininthispaper}
Let $\lambda$ and $\mu$ be dominant integral weights, and $v, w$ elements of $W$. Each connected component of $\mathcal{B}_v(\lambda) \otimes \mathcal{B}_w(\mu)$ is isomorphic to some Demazure crystal if and only if $\lfloor v \rfloor^\lambda \in W_{\lceil w \rceil^\mu}$.
\end{thm}

Before proving Theorem \ref{mainthm_1}, Corollary \ref{maincor}, and Theorem \ref{mainthm_4}, we give some examples of tensor products of Demazure crystals.

\subsection{Examples of tensor products of Demazure crystals}
Here we give some examples. In this subsection, we consider the case $\mathfrak{g} = \mathfrak{sl}_3$. Let $I := \{1, 2\}$ and $\mathfrak{h} := \{ h \in \mathfrak{g} \ | \ h \text{ is a diagonal matrix} \}$. For $k \in \{ 1, 2, 3\}$, we define the maps $\varepsilon_k : \mathfrak{h} \rightarrow \mathbb{C}$ by
\begin{align*}
\varepsilon_k \left( \begin{pmatrix} z_1 & 0 & 0 \\ 0 & z_2 & 0 \\ 0 & 0 & z_3 \end{pmatrix} \right) := z_k.
\end{align*}

Let $\alpha_i := \varepsilon_i - \varepsilon_{i+1}$ for $i \in I$, and define $\varpi_1$ and $\varpi_2 \in \mathfrak{h}^\ast$ by $\varpi_1 := \varepsilon_1$ and $\varpi_2 := \varepsilon_1 + \varepsilon_2$, respectively. Then $\{ \alpha_1, \alpha_2 \}$ is the set of simple roots, and $\{\varpi_1, \varpi_2\}$ is the set of fundamental weights. In this case, for $i \in I$ the simple reflection $s_i$ is the linear transformation on $\mathfrak{h}^\ast$, which satisfy the following equality for $k \in \{1, 2, 3\}$.
\begin{align*}
s_i (\varepsilon_k) = \begin{cases} \varepsilon_{i+1} & (k = i), \\ \varepsilon_i & (k = i+1), \\ \varepsilon_k & (k \not= i, i+1). \end{cases}
\end{align*}

\begin{exm}
Let $v = s_1 s_2$, $w = s_1 s_2 s_1$, $\lambda = \varpi_1 + \varpi_2$, and $\mu = \varpi_1$. Then we can verify the following equation by direct calculation.
\begin{align}\label{exm1}
\mathcal{B}_v(\lambda) \otimes \mathcal{B}_w(\mu) \simeq \mathcal{B}_{s_1 s_2}( 2\varpi_1 + \varpi_2 ) \sqcup \mathcal{B}_{s_1 s_2} ( 2\varpi_2 ) \sqcup \mathcal{B}_{s_1} ( \varpi_1 ).
\end{align}
In this case, $\lfloor v \rfloor^\lambda = v$ is contained in $W_{\lceil w \rceil^\mu} = W_w$ since $w$ is the longest element of $W$.

We compute $u(\pi, v)$ for $\pi \in \mathcal{B}_w(\mu)^\lambda$ and compare (\ref{exm1}) with the isomorphism in Corollary \ref{maincor}. In our case, we have $\mathcal{B}_w(\mu) = \{ \pi^{\varepsilon_1}, \pi^{\varepsilon_2}, \pi^{\varepsilon_3} \}$.

First, we consider $\pi^{\varepsilon_1}$. For $t \in [0, 1]$, we have $\lambda + \pi^{\varepsilon_1}(t) = (t + 2)\varepsilon_1 + \varepsilon_2$, and hence $\langle \lambda + \pi^{\varepsilon_1}(t), \alpha_1^\vee \rangle = t + 1 > 0$, $\langle \lambda + \pi^{\varepsilon_1}(t), \alpha_2^\vee \rangle = 1 > 0$. It follows that $\pi^{\varepsilon_1} \in \mathcal{B}_w(\mu)^\lambda$ and $W(t) = \{ e \}$. Therefore, $w(\pi^{\varepsilon_1}) = e$ and $u(\pi^{\varepsilon_1}, v) = s_1 s_2$. Note that $\lambda + \wt(\pi^{\varepsilon_1}) = 3 \varepsilon_1 + \varepsilon_2 = 2\varpi_1 + \varpi_2$.

Next, we consider $\pi^{\varepsilon_2}$. Fix $t \in [0, 1]$. Then we have $\lambda + \pi^{\varepsilon_2}(t) = 2 \varepsilon_1 + (t + 1)\varepsilon_2$, and hence $\langle \lambda + \pi^{\varepsilon_2}(t), \alpha_1^\vee \rangle = -t + 1 \geq 0$, $\langle \lambda + \pi^{\varepsilon_2}(t), \alpha_2^\vee \rangle = t + 1 >0$. It follows that $\pi^{\varepsilon_2} \in \mathcal{B}_w(\mu)^\lambda$. Since $\langle \lambda + \pi^{\varepsilon_2}(t), \alpha_1^\vee \rangle = 0$ if and only if $t = 1$, we have
\begin{align*}
W(t) = \begin{cases} \{ e \} & (0 \leq t < 1), \\ \{ e, s_1 \} & (t = 1). \end{cases}
\end{align*}
Therefore, we conclude that $w(\pi^{\varepsilon_2}) = s_1$ and $u(\pi^{\varepsilon_2}, v) = s_1 s_2 s_1$. Note that $\lambda + \wt(\pi^{\varepsilon_2}) = 2\varepsilon_1 + 2\varepsilon_2 = 2\varpi_2$.

Finally, we consider $\pi^{\varepsilon_3}$. Take $t \in [0, 1]$. Then one has $\lambda + \pi^{\varepsilon_3}(t) = 2\varepsilon_1 + \varepsilon_2 + t\varepsilon_3$, and hence $\langle \lambda + \pi^{\varepsilon_3}(t), \alpha_1^\vee \rangle = 1 >0$, $\langle \lambda + \pi^{\varepsilon_3}(t), \alpha_2^\vee \rangle = -t + 1 \geq 0$. It follows that $\pi^{\varepsilon_3} \in \mathcal{B}_w(\mu)^\lambda$. Since $\langle \lambda + \pi^{\varepsilon_3}(t), \alpha_2^\vee \rangle = 0$ if and only if $t = 1$, we have
\begin{align*}
W(t) = \begin{cases} \{ e \} & (0 \leq t < 1), \\ \{e, s_2 \} & (t = 1). \end{cases}
\end{align*}
Therefore, we can verify that $w(\pi^{\varepsilon_3}) = s_2$ and $u(\pi^{\varepsilon_3}, v) = s_1 s_2$. Note that $\lambda + \wt(\pi^{\varepsilon_3}) = 2\varepsilon_1 + \varepsilon_2 + \varepsilon_3 = \varepsilon_1 = \varpi_1$. Recall that $\varepsilon_1 + \varepsilon_2 + \varepsilon_3 = 0$ since we are considering $\mathfrak{g} = \mathfrak{sl}_3$.

Thus we have
\begin{align*}
& \ \bigsqcup_{\pi \in \mathcal{B}_w(\mu)^\lambda} \mathcal{B}_{u(\pi, v)}(\lambda + \wt(\pi)) \\
=& \ \mathcal{B}_{u(\pi^{\varepsilon_1}, v)}(\lambda + \wt(\pi^{\varepsilon_1})) \sqcup \mathcal{B}_{u(\pi^{\varepsilon_2}, v)}(\lambda + \wt(\pi^{\varepsilon_2})) \sqcup \mathcal{B}_{u(\pi^{\varepsilon_3}, v)}(\lambda + \wt(\pi^{\varepsilon_3})) \\
=& \ \mathcal{B}_{s_1 s_2}(2\varpi_1 + \varpi_2) \sqcup \mathcal{B}_{s_1 s_2 s_1}(2\varpi_2) \sqcup \mathcal{B}_{s_1 s_2}(\varpi_1) \\
=& \ \mathcal{B}_{s_1 s_2}(2\varpi_1 + \varpi_2) \sqcup \mathcal{B}_{s_1 s_2}(2\varpi_2) \sqcup \mathcal{B}_{s_1}(\varpi_1) \\
\simeq& \ \mathcal{B}_v(\lambda) \otimes \mathcal{B}_w(\mu).
\end{align*}

\end{exm}

\begin{exm}
Let $v = s_1$, $w = s_1 s_2$, $\lambda = 2 \varpi_1 + \varpi_2$, and $\mu = \varpi_1 + 2 \varpi_2$. Then we can verify the following equation by direct calculation.
\begin{align}\label{exm2}
& \ \mathcal{B}_v(\lambda) \otimes \mathcal{B}_w(\mu) \nonumber \\ 
\simeq & \ \mathcal{B}_{s_1}( 3\varpi_1 + 3\varpi_2 ) \sqcup \mathcal{B}_{s_1} ( \varpi_1 + 4\varpi_2 ) \sqcup \mathcal{B}_{s_1 s_2} ( 4\varpi_1 + \varpi_2 ) \nonumber \\
& \ \sqcup \mathcal{B}_{s_1}( 2\varpi_1 + 2\varpi_2 ) \sqcup \mathcal{B}_{e} (3\varpi_2 ) \sqcup \mathcal{B}_{s_1} ( 3\varpi_1)  \sqcup \mathcal{B}_{s_1} ( \varpi_1 + \varpi_2 ). 
\end{align}
In this case, $\lfloor v \rfloor^\lambda = v$ is contained in $W_{\lceil w \rceil^\mu} = W_w$ since $\ell(s_1 w) = \ell(s_2) = 1 < 2 = \ell(w)$.

We compare (\ref{exm2}) with the isomorphism in Corollary \ref{maincor}. Set
\begin{align*}
\pi_1 := & \ \pi^\mu, \\
\pi_2 := & \ \pi^{- \varpi_1 + 3 \varpi_2}, \\
\pi_3 := & \ \left( 3 \varpi_1 - 2 \varpi_2, \varpi_1 + 2 \varpi_2; 0, \frac{1}{2}, 1 \right) , \\
\pi_4 := & \ \left( -3 \varpi_1 + \varpi_2, 3 \varpi_1 - 2 \varpi_2, \varpi_1 + 2 \varpi_2; 0, \frac{1}{3}, \frac{1}{2}, 1 \right), \\
\pi_5 := & \ \left( -3 \varpi_1 + \varpi_2, - \varpi_1 + 3 \varpi_2; 0, \frac{1}{2}, 1 \right), \\
\pi_6 := & \ \left( -3 \varpi_1 + \varpi_2, 3 \varpi_1 - 2 \varpi_2; 0, \frac{1}{3}, 1 \right), \\
\pi_7 := & \ \left( -3 \varpi_1 + \varpi_2, 3 \varpi_1 - 2 \varpi_2; 0, \frac{2}{3}, 1 \right).
\end{align*}
Then we have $\mathcal{B}_w(\mu)^\lambda = \{ \pi_1, \ldots, \pi_7 \}$. We can verify that $u(\pi_1, v) = s_1$, $u(\pi_2, v) = s_1$, $u(\pi_3, v) = s_1 s_2$, $u(\pi_4, v) = s_1$, $u(\pi_5, v) = s_1$, $u(\pi_6, v) = s_1 s_2$, and $u(\pi_7, v) = s_1$. Thus we have
\begin{align*}
& \ \bigsqcup_{\pi \in \mathcal{B}_w(\mu)^\lambda} \mathcal{B}_{u(\pi, v)}(\lambda + \wt(\pi)) \\
=& \ \bigsqcup_{k \in \{ 1, \ldots, 7 \}} \mathcal{B}_{u(\pi_k, v)}(\lambda + \wt(\pi_k)) \\
=& \ \mathcal{B}_{s_1}( 3\varpi_1 + 3\varpi_2 ) \sqcup \mathcal{B}_{s_1} ( \varpi_1 + 4\varpi_2 ) \sqcup \mathcal{B}_{s_1 s_2} ( 4\varpi_1 + \varpi_2 )  \\
& \ \sqcup \mathcal{B}_{s_1}( 2\varpi_1 + 2\varpi_2 ) \sqcup \mathcal{B}_{s_1} ( 3\varpi_2 ) \sqcup \mathcal{B}_{s_1 s_2} ( 3\varpi_1 ) \sqcup \mathcal{B}_{s_1} ( \varpi_1 + \varpi_2 ) \\
= & \ \mathcal{B}_{s_1}( 3\varpi_1 + 3\varpi_2 ) \sqcup \mathcal{B}_{s_1} ( \varpi_1 + 4\varpi_2 ) \sqcup \mathcal{B}_{s_1 s_2} ( 4\varpi_1 + \varpi_2 ) \nonumber \\
& \ \sqcup \mathcal{B}_{s_1}( 2\varpi_1 + 2\varpi_2 ) \sqcup \mathcal{B}_{e} (3\varpi_2 ) \sqcup \mathcal{B}_{s_1} ( 3\varpi_1 ) \sqcup \mathcal{B}_{s_1} ( \varpi_1 + \varpi_2 ) \\
\simeq & \ \mathcal{B}_v(\lambda) \otimes \mathcal{B}_w(\mu).
\end{align*}

\end{exm}

\begin{exm}
Let $v = w = s_1 s_2$, $\lambda = \varpi_1 + \varpi_2$, and $\mu = \varpi_1$. One of the connected components of the crystal graph of $\mathcal{B}_v(\lambda) \otimes \mathcal{B}_w(\mu)$ is as follows:

\begin{center}
\begin{picture}(29.2,3)(0,0)
\put(0,1){$\pi^\lambda \ast \tilde{f}_1(\pi^\mu)$}
\put(6,1){\vector(1,0){3}}
\put(7,1.5){2}
\put(10,1){$\tilde{f}_2(\pi^\lambda) \ast \tilde{f}_1(\pi^\mu)$}
\put(17.8,1){\vector(1,0){3}}
\put(18.8,1.5){1}
\put(21.8,1){$\tilde{f}_1\tilde{f}_2(\pi^\lambda) \ast \tilde{f}_1(\pi^\mu)$.}
\end{picture}
\end{center}
The highest weight of this connected component is $2 \varpi_2$. Therefore, if this connected component were a Demazure crystal, it must be $\mathcal{B}_{s_1 s_2} ( 2 \varpi_2 )$. However clearly it is not. Hence $\mathcal{B}_v(\lambda) \otimes \mathcal{B}_w(\mu)$ is not a disjoint union of Demazure crystals. This example shows that in general $\mathcal{B}_v(\lambda) \otimes \mathcal{B}_w(\mu)$ is not  a disjoint union of Demazure crystals for $v, w \in W$ and $\lambda, \mu \in P^+$. Note that in this case, we have $\lfloor v \rfloor^\lambda = v$ and $\lceil w \rceil^\mu = w$; thus, $\lfloor v \rfloor^\lambda$ is not contained in $W_{\lceil w \rceil^\mu}$.
\end{exm}

\subsection{Proofs of Theorem \ref{mainthm_1} and Corollary \ref{maincor}}

First of all, we prove the following two lemmas.

\begin{lem}\label{mainthm_2}
Let $\lambda$ and $\mu$ be dominant integral weights, and $v$ and $w$ elements of $W$, and take a reduced expression $v = s_{i_1} \cdots s_{i_l}$. Then we have $\mathcal{B}_v(\lambda) \otimes \mathcal{B}_w(\mu) \subset \mathcal{B}_{(i_1, \ldots, i_l), w, \lambda, \mu}$.
\end{lem}

\begin{lem}\label{mainthm_3}
Let $\lambda$ and $\mu$ be dominant integral weights, $w$ an element of $W$, and $v$ an element of $W_w$ and take a reduced expression $v = s_{i_1} \cdots s_{i_l}$. Then we have $\mathcal{B}_v(\lambda) \otimes \mathcal{B}_w(\mu) \supset \mathcal{B}_{(i_1, \ldots, i_l), w, \lambda, \mu}$.
\end{lem}

Note that the assumption of Lemma \ref{mainthm_2} is weaker than that of Lemma \ref{mainthm_3}.

\begin{proof}[Proof of Lemma \ref{mainthm_2}]
Take $b \otimes b^\prime \in \mathcal{B}_v(\lambda) \otimes \mathcal{B}_w(\mu)$ arbitrarily. Since $b \otimes b^\prime \not = 0$, it is sufficient to show that $b \otimes b^\prime \in \bigcup_{a_1, \ldots, a_l \geq 0} \tilde{f}_{i_1}^{a_1} \cdots \tilde{f}_{i_l}^{a_l} \left( \mathcal{B}_e(\lambda) \otimes \mathcal{B}_w(\mu) \right)$. For this, we show that $\tilde{e}_{i_l}^{\text{max}} \cdots \tilde{e}_{i_1}^{\text{max}} \left( b \otimes b^\prime \right) \in \mathcal{B}_e(\lambda) \otimes \mathcal{B}_w(\mu)$.

Let $b = \tilde{f}_{i_1}^{a_1} \cdots \tilde{f}_{i_l}^{a_l} (b_\lambda)$ be the string parametrization for $b$. By the tensor product rule, there is some $r_1 \in \mathbb{Z}$ with $r_1 \geq 0$ such that
\begin{align*}
\tilde{e}_{i_1}^{\text{max}} \left( b \otimes b^\prime \right) &= \tilde{e}_{i_1}^{\text{max}} \left( \tilde{f}_{i_1}^{a_1} \cdots \tilde{f}_{i_l}^{a_l} (b_\lambda) \otimes b^\prime \right) \\
&= \tilde{f}_{i_2}^{a_2} \cdots \tilde{f}_{i_l}^{a_l} (b_\lambda) \otimes \tilde{e}_{i_1}^{r_1} (b^\prime)
\end{align*}
since $\tilde{e}_{i_1} \tilde{f}_{i_2}^{a_2} \cdots \tilde{f}_{i_l}^{a_l} (b_\lambda) = 0$. Similarly, there is a nonnegative integer $r_2$ such that
\begin{align*}
\tilde{e}_{i_2}^{\text{max}}\tilde{e}_{i_1}^{\text{max}} \left( b \otimes b^\prime \right) &= \tilde{f}_{i_3}^{a_3} \cdots \tilde{f}_{i_l}^{a_l} (b_\lambda) \otimes \tilde{e}_{i_2}^{r_2} \tilde{e}_{i_1}^{r_1} (b^\prime).
\end{align*}
Repeating this, we obtain nonnegative integers $r_1, \ldots, r_l$ such that
\begin{align*}
\tilde{e}_{i_l}^{\text{max}} \cdots \tilde{e}_{i_1}^{\text{max}} \left( b \otimes b^\prime \right) &= b_\lambda \otimes \tilde{e}_{i_l}^{r_l} \cdots \tilde{e}_{i_1}^{r_1} (b^\prime).
\end{align*}
Since a Demazure crystal is stable under raising Kashiwara operators, we have $\tilde{e}_{i_l}^{r_l} \cdots \tilde{e}_{i_1}^{r_1} (b^\prime) \in \mathcal{B}_w(\mu) \sqcup \{ 0 \}$. Since $\tilde{e}_{i_l}^{\text{max}} \cdots \tilde{e}_{i_1}^{\text{max}} \left( b \otimes b^\prime \right)$ is not $0$, the element $\tilde{e}_{i_l}^{r_l} \cdots \tilde{e}_{i_1}^{r_1} (b^\prime)$ is not $0$. Therefore, we conclude that the element $\tilde{e}_{i_l}^{\text{max}} \cdots \tilde{e}_{i_1}^{\text{max}} \left( b \otimes b^\prime \right)$ belongs to $\mathcal{B}_e(\lambda) \otimes \mathcal{B}_w(\mu)$.
\end{proof}

\begin{proof}[Proof of Lemma \ref{mainthm_3}]
Take $b \in \mathcal{B}_w(\mu)$. Fix $k \in \{ 1, \ldots, l \}$. Since $\ell(s_{i_k}w) < \ell(w)$, we have $\tilde{f}_{i_k} (b) \in \mathcal{B}_w(\mu) \sqcup \{ 0 \}$ by Lemma \ref{f(Dem)}(2).

Let us prove that for all $b \in \mathcal{B}_w(\mu)$ and nonnegative integers $a_1, \ldots, a_l$, one has $\tilde{f}_{i_1}^{a_1} \cdots \tilde{f}_{i_l}^{a_l} \left( b_\lambda \otimes b \right) \in \mathcal{B}_v(\lambda) \otimes \mathcal{B}_w(\mu)$ if $\tilde{f}_{i_1}^{a_1} \cdots \tilde{f}_{i_l}^{a_l} \left( b_\lambda \otimes b \right)$ is defined. Now assume that $\tilde{f}_{i_1}^{a_1} \cdots \tilde{f}_{i_l}^{a_l} \left( b_\lambda \otimes b \right) \not= 0$. By the tensor product rule, there is nonnegative integers $c_l, d_l$ such that
\begin{align*}
\tilde{f}_{i_l}^{a_l} \left( b_\lambda \otimes b \right) = \tilde{f}_{i_l}^{c_l} (b_\lambda) \otimes \tilde{f}_{i_l}^{d_l} (b).
\end{align*}
By the discussion above, we have $\tilde{f}_{i_l}^{d_l} (b) \in \mathcal{B}_w(\mu)$ since $\tilde{f}_{i_1}^{a_1} \cdots \tilde{f}_{i_l}^{a_l} \left( b_\lambda \otimes b \right)$ is not $0$. In the same way, we obtain nonnegative integers $c_1, \ldots, c_l$ and $d_1, \ldots, d_l$ such that
\begin{align*}
\tilde{f}_{i_1}^{a_1} \cdots \tilde{f}_{i_l}^{a_l} \left( b_\lambda \otimes b \right) = \tilde{f}_{i_1}^{c_1} \cdots \tilde{f}_{i_l}^{c_l} (b_\lambda) \otimes \tilde{f}_{i_1}^{d_1} \cdots \tilde{f}_{i_l}^{d_l} (b),
\end{align*}
with $\tilde{f}_{i_1}^{d_1} \cdots \tilde{f}_{i_l}^{d_l} (b) \in \mathcal{B}_w(\mu)$. Since $(i_1, \ldots, i_l)$ is a reduced word for $v$ and since $\tilde{f}_{i_1}^{a_1} \cdots \tilde{f}_{i_l}^{a_l} \left( b_\lambda \otimes b \right) \not= 0$, one has $\tilde{f}_{i_1}^{c_1} \cdots \tilde{f}_{i_l}^{c_l} (b_\lambda) \in \mathcal{B}_v(\lambda)$. Therefore, we conclude that $\tilde{f}_{i_1}^{a_1} \cdots \tilde{f}_{i_l}^{a_l} \left( b_\lambda \otimes b \right)$ is contained in $\mathcal{B}_v(\lambda) \otimes \mathcal{B}_w(\mu)$.
\end{proof}

Combining these lemmas, we obtain the following assertion.

\begin{lem}\label{mainthm_5}
Let $\lambda$ and $\mu$ be dominant integral weights, $w$ an element of $W$, and $v$ an element of $W_w$, and take a reduced expression $v = s_{i_1} \cdots s_{i_l}$. Then we have $\mathcal{B}_v(\lambda) \otimes \mathcal{B}_w(\mu) = \mathcal{B}_{(i_1, \ldots, i_l), w, \lambda, \mu}$.
\end{lem}

\begin{proof}[Proof of Theorem \ref{mainthm_1}]
Since $\mathcal{B}_v(\lambda) = \mathcal{B}_{\lfloor v \rfloor^\lambda}(\lambda)$ and $\mathcal{B}_w(\mu) = \mathcal{B}_{\lceil w \rceil^\mu}(\mu)$, we obtain the desired assertion by Lemma \ref{mainthm_5}. More precisely, if we take a reduced expression $\lfloor v \rfloor^\lambda = s_{i_1} \cdots s_{i_l}$, then we have
\begin{align*}
\mathcal{B}_v(\lambda) \otimes \mathcal{B}_w(\mu) &= \mathcal{B}_{\lfloor v \rfloor^\lambda}(\lambda) \otimes \mathcal{B}_{\lceil w \rceil^\mu}(\mu) \\
&= \bigcup_{a_1, \ldots, a_l \geq 0} \tilde{f}_{i_1}^{a_1} \cdots \tilde{f}_{i_l}^{a_l} \left( \mathcal{B}_e(\lambda) \otimes \mathcal{B}_{\lceil w \rceil^\mu}(\mu) \right) \setminus \{ 0 \} \\
&= \bigcup_{a_1, \ldots, a_l \geq 0} \tilde{f}_{i_1}^{a_1} \cdots \tilde{f}_{i_l}^{a_l} \left( \mathcal{B}_e(\lambda) \otimes \mathcal{B}_{w}(\mu) \right) \setminus \{ 0 \} \\
&= \mathcal{B}_{(i_1, \ldots, i_l), w, \lambda, \mu}.
\end{align*}
\end{proof}
Now we prove Corollary \ref{maincor}.

\begin{proof}[Proof of Corollary \ref{maincor}]
Fix a reduced decomposition $\lfloor v \rfloor^\lambda = s_{i_1} \cdots s_{i_l}$. By Theorem \ref{mainthm_1} and Theorem \ref{1tensDem}, we have
\begin{align}
\mathcal{B}_v(\lambda) \otimes \mathcal{B}_w(\mu) &= \bigcup_{a_1, \ldots, a_l \geq 0} \tilde{f}_{i_1}^{a_1} \cdots \tilde{f}_{i_l}^{a_l} \left( \mathcal{B}_e(\lambda) \otimes \mathcal{B}_w(\mu) \right) \setminus \{ 0 \} \nonumber \\
&= \bigcup_{a_1, \ldots, a_l \geq 0} \tilde{f}_{i_1}^{a_1} \cdots \tilde{f}_{i_l}^{a_l} \left( \bigsqcup_{\pi \in \mathcal{B}_w(\mu)^\lambda} C(\pi, e) \right) \setminus \{ 0 \} \nonumber \\
&= \bigcup_{\pi \in \mathcal{B}_w(\mu)^\lambda} \left( \bigcup_{a_1, \ldots, a_l \geq 0} \tilde{f}_{i_1}^{a_1} \cdots \tilde{f}_{i_l}^{a_l} ( C(\pi, e) ) \setminus \{ 0 \}\right). \label{union}
\end{align}
Since the subset $\bigcup_{a_1, \ldots, a_l \geq 0} \tilde{f}_{i_1}^{a_1} \cdots \tilde{f}_{i_l}^{a_l} ( C(\pi, e) ) \setminus \{ 0 \}$ of $\mathcal{B}_v(\lambda) \otimes \mathcal{B}_w(\mu)$ is connected and it contains $\pi^\lambda \otimes \pi$, one has 
\begin{align*}
\bigcup_{a_1, \ldots, a_l \geq 0} \tilde{f}_{i_1}^{a_1} \cdots \tilde{f}_{i_l}^{a_l} ( C(\pi, e) ) \setminus \{ 0 \} \subset C(\pi, v).
\end{align*}
Hence the union in (\ref{union}) is disjoint, that is, we have
\begin{align*}
\mathcal{B}_v(\lambda) \otimes \mathcal{B}_w(\mu) &= \bigsqcup_{\pi \in \mathcal{B}_w(\mu)^\lambda} \left( \bigcup_{a_1, \ldots, a_l \geq 0} \tilde{f}_{i_1}^{a_1} \cdots \tilde{f}_{i_l}^{a_l} ( C(\pi, e) ) \setminus \{ 0 \}\right).
\end{align*}

Assume that
\begin{align*}
\bigcup_{a_1, \ldots, a_l \geq 0} \tilde{f}_{i_1}^{a_1} \cdots \tilde{f}_{i_l}^{a_l} ( C(\pi, e) ) \setminus \{ 0 \} \subsetneq C(\pi, v)
\end{align*}
for some $\pi \in \mathcal{B}_w(\mu)^\lambda$. Then we see that
\begin{align*}
\mathcal{B}_v(\lambda) \otimes \mathcal{B}_w(\mu) &= \bigsqcup_{\pi \in \mathcal{B}_w(\mu)^\lambda} \left( \bigcup_{a_1, \ldots, a_l \geq 0} \tilde{f}_{i_1}^{a_1} \cdots \tilde{f}_{i_l}^{a_l} ( C(\pi, e) ) \setminus \{ 0 \}\right) \\
&\subsetneq \bigsqcup_{\pi \in \mathcal{B}_w(\mu)^\lambda} C(\pi, v) \\
&= \mathcal{B}_v(\lambda) \otimes \mathcal{B}_w(\mu),
\end{align*}
which is a contradiction. Thus, for all $\pi \in \mathcal{B}_w(\mu)^\lambda$, we have
\begin{align*}
\bigcup_{a_1, \ldots, a_l \geq 0} \tilde{f}_{i_1}^{a_1} \cdots \tilde{f}_{i_l}^{a_l} ( C(\pi, e) ) \setminus \{ 0 \} = C(\pi, v).
\end{align*}

Let $\pi \in \mathcal{B}_w(\mu)^\lambda$. Since $C(\pi, e) \simeq \mathcal{B}_{w(\pi)}(\lambda + \wt (\pi))$ by Theorem \ref{1tensDem}, it follows that
\begin{align*}
C(\pi, v) &= \bigcup_{a_1, \ldots, a_l \geq 0} \tilde{f}_{i_1}^{a_1} \cdots \tilde{f}_{i_l}^{a_l} ( C(\pi, e) ) \setminus \{ 0 \} \\
&\simeq \bigcup_{a_1, \ldots, a_l \geq 0} \tilde{f}_{i_1}^{a_1} \cdots \tilde{f}_{i_l}^{a_l} ( \mathcal{B}_{w(\pi)}(\lambda + \wt (\pi)) ) \setminus \{ 0 \}.
\end{align*}
Hence, to prove the corollary, it is sufficient to show that 
\begin{align*}
\bigcup_{a_1, \ldots, a_l \geq 0} \tilde{f}_{i_1}^{a_1} \cdots \tilde{f}_{i_l}^{a_l} ( \mathcal{B}_{w(\pi)}(\lambda + \wt(\pi)) ) \setminus \{ 0 \} = \mathcal{B}_{u(\pi, v)}(\lambda + \wt(\pi))
\end{align*}
for all $\pi \in \mathcal{B}_w(\mu)^\lambda$. Let $\pi \in \mathcal{B}_w(\mu)^\lambda$. By Lemma \ref{f(Dem)}, we have
\begin{align*}
& \ \bigcup_{a_l \geq 0} \tilde{f}_{i_l}^{a_l} ( \mathcal{B}_{w(\pi)}(\lambda + \wt(\pi)) ) \setminus \{ 0 \}, \\
=& \ \begin{cases} \mathcal{B}_{s_{i_l}w(\pi)}(\lambda + \wt(\pi)) & \text{ if } \ell (s_{i_l}w(\pi)) > \ell (w(\pi)) \\ \mathcal{B}_{w(\pi)}(\lambda + \wt(\pi)) & \text{ if } \ell (s_{i_l}w(\pi)) < \ell (w(\pi)) \end{cases} \\
=& \ \mathcal{B}_{u_l}(\lambda + \wt(\pi)).
\end{align*}

Now we proceed by induction. Suppose that
\begin{align*}
\bigcup_{a_j, \ldots, a_l \geq 0} \tilde{f}_{i_j}^{a_j} \cdots \tilde{f}_{i_l}^{a_l} ( \mathcal{B}_{w(\pi)}(\lambda + \wt(\pi)) ) \setminus \{ 0 \} = \mathcal{B}_{u_j}(\lambda + \wt(\pi))
\end{align*}
for $j \in \{2, \ldots, l \}$. Again by Lemma \ref{f(Dem)}, one has
\begin{align*}
& \ \bigcup_{a_{j-1}, \ldots, a_l \geq 0} \tilde{f}_{i_{j-1}}^{a_{j-1}} \cdots \tilde{f}_{i_l}^{a_l} ( \mathcal{B}_{w(\pi)}(\lambda + \wt(\pi)) ) \setminus \{ 0 \} \\
=& \ \bigcup_{a_{j-1} \geq 0} \tilde{f}_{i_{j-1}}^{a_{j-1}} \left( \bigcup_{a_j, \ldots, a_l \geq 0} \tilde{f}_{i_j}^{a_j} \cdots \tilde{f}_{i_l}^{a_l} ( \mathcal{B}_{w(\pi)}(\lambda + \wt(\pi)) ) \setminus \{ 0 \}\right) \setminus \{ 0\} \\
=& \ \bigcup_{a_{j-1} \geq 0} \tilde{f}_{i_{j-1}}^{a_{j-1}} (\mathcal{B}_{u_j}(\lambda + \wt(\pi))) \setminus \{ 0\} \\
=& \ \begin{cases} \mathcal{B}_{s_{i_{j-1}}u_j} (\lambda + \wt(\pi)) & \text{ if } \ell(s_{i_{j-1}}u_j) > \ell(u_j), \\ \mathcal{B}_{u_j} (\lambda + \wt(\pi)) & \text{ if } \ell(s_{i_{j-1}}u_j) < \ell(u_j)\end{cases} \\
=& \ \mathcal{B}_{u_{j-1}}(\lambda+\wt(\pi)).
\end{align*}

Therefore, we conclude that
\begin{align*}
& \ \bigcup_{a_1, \ldots, a_l \geq 0} \tilde{f}_{i_1}^{a_1} \cdots \tilde{f}_{i_l}^{a_l} ( \mathcal{B}_{w(\pi)}(\lambda + \wt(\pi)) ) \setminus \{ 0 \} \\ =& \ \mathcal{B}_{u_1}(\lambda + \wt(\pi)) \\
=& \ \mathcal{B}_{u(\pi, v)}(\lambda + \wt(\pi)).
\end{align*}
\end{proof}

\subsection{Proof of Theorem \ref{mainthm_4}}

The following lemma is important.

\begin{lem}\label{key}
For $\mu \in P^+$, $w \in W$, and $i \in I$, $\ell(s_{i} \lceil w \rceil^\mu) > \ell(\lceil w \rceil^\mu)$ holds if and only if $\langle w\mu , \alpha_{i}^\vee \rangle > 0$.
\end{lem}

\begin{proof}
First of all, observe that $\langle w\mu , \alpha_{i}^\vee \rangle = \langle \lceil w \rceil^\mu \mu , \alpha_{i}^\vee \rangle = \langle \mu , (\lceil w \rceil^\mu)^{-1} \alpha_{i}^\vee \rangle$. 

Assume that $\ell(s_{i} \lceil w \rceil^\mu) > \ell(\lceil w \rceil^\mu)$. Then, $(\lceil w \rceil^\mu)^{-1} \alpha_{i}^\vee$ is a positive coroot by Corollary \ref{keycor}. Since $\mu$ is a dominant integral weight, we have $\langle \mu , (\lceil w \rceil^\mu)^{-1} \alpha_{i}^\vee \rangle \geq 0$. Hence $\langle w\mu , \alpha_{i}^\vee \rangle \geq 0$. Now suppose that $\langle w\mu , \alpha_i^\vee \rangle = 0$. Then $\langle \lceil w \rceil^\mu \mu , \alpha_i^\vee \rangle = \langle w\mu , \alpha_i^\vee \rangle = 0$, and hence $s_i \lceil w \rceil^\mu \mu = \lceil w \rceil^\mu \mu = w\mu$. This implies that $s_i \lceil w \rceil^\mu \in wW_\mu$. Since $\ell(s_{i} \lceil w \rceil^\mu) > \ell(\lceil w \rceil^\mu)$, this contradicts the fact that $\lceil w \rceil^\mu$ is maximal in $wW_\mu$. Thus $\langle w\mu, \alpha_i^\vee \rangle$ is not equal to $0$. Hence one has $\langle w\mu, \alpha_i^\vee \rangle > 0$.

Next, assume that $\langle w\mu, \alpha_i^\vee \rangle > 0$. Then $\langle \mu , (\lceil w \rceil^\mu)^{-1} \alpha_{i}^\vee \rangle > 0$. Since $\mu$ is a dominant integral weight, $(\lceil w \rceil^\mu)^{-1} \alpha_{i}^\vee$ must be a positive coroot. By Lemma \ref{length1} and Corollary \ref{length2}, we have $\ell(s_{i} \lceil w \rceil^\mu) = \ell(\lceil w \rceil^\mu) + 1 > \ell(\lceil w \rceil^\mu)$.
\end{proof}

\begin{proof}[Proof of Theorem \ref{mainthm_4}]
Fix a reduced expression $\lfloor v \rfloor^{\lambda} = s_{i_1} \cdots s_{i_l}$. Since $\lfloor v \rfloor^\lambda \not\in W_{\lceil w \rceil^\mu}$ by the assumption, there exists $k \in \{1, \ldots, l \}$ such that $\ell(s_{i_k} \lceil w \rceil^\mu) > \ell(\lceil w \rceil^\mu)$. For this $k$, we consider $\pi := \pi^{s_{i_{k+1}} \cdots s_{i_l} \lambda} \otimes \pi^{w\mu} \in \mathcal{B}_v(\lambda) \otimes \mathcal{B}_w(\mu)$. Let $C$ denote the connected component of $\mathcal{B}_v(\lambda) \otimes \mathcal{B}_w(\mu)$ containing $\pi$.

By Lemma \ref{key}, one has $\langle w\mu, \alpha_{i_k}^\vee \rangle >0$, which shows that the height function $h_{i_k}^{\pi^{w\mu}}(t) = \langle \pi^{w\mu}(t), \alpha_{i_k}^\vee \rangle$ of the straight-line path $\pi^{w\mu}$ is strictly increasing on $[0, 1]$. Hence $\tilde{f}_{i_k}(\pi^{w\mu}) \not= 0$ and $\tilde{e}_{i_k}(\pi^{w\mu}) = 0$. Since $w\mu$ is the lowest weight of $\mathcal{B}_w(\mu)$ and $\wt (\tilde{f}_{i_k}(\pi^{w\mu})) = w\mu - \alpha_{i_k}$ is less than $w\mu$, $\tilde{f}_{i_k}(\pi^{w\mu}) \not\in \mathcal{B}_w(\mu) \sqcup \{ 0 \}$.

Now we prove that $\tilde{f}_{i_k} (\pi^{s_{i_{k+1}} \cdots s_{i_l} \lambda}) \in \mathcal{B}_v(\lambda)$. First we show that the element $\tilde{f}_{i_k}(\pi^{s_{i_{k+1}} \cdots s_{i_l} \lambda})$ is not $0$.
Suppose that $\langle s_{i_{k+1}} \cdots s_{i_l} \lambda, \alpha_{i_k}^\vee \rangle \leq 0$. If $\langle s_{i_{k+1}} \cdots s_{i_l} \lambda, \alpha_{i_k}^\vee \rangle < 0$, then $\langle \lambda, (s_{i_{k+1}} \cdots s_{i_l} )^{-1} \alpha_{i_k}^\vee \rangle < 0$. Since $\lambda$ is a dominant integral weight, $(s_{i_{k+1}} \cdots s_{i_l} )^{-1} \alpha_{i_k}^\vee$ must be a negative coroot. By Lemma \ref{length1} and Corollary \ref{length2}, we have $\ell(s_{i_k}s_{i_{k+1}} \cdots s_{i_l}) < \ell(s_{i_{k+1}} \cdots s_{i_l})$, which contradicts the fact that $s_{i_1} \cdots s_{i_l}$ is reduced. Thus, $\langle s_{i_{k+1}} \cdots s_{i_l} \lambda, \alpha_{i_k}^\vee \rangle$ must be $0$. In this case, $s_{i_k}s_{i_{k+1}} \cdots s_{i_l} \lambda = s_{i_{k+1}} \cdots s_{i_l} \lambda$. Therefore,
\begin{align*}
v\lambda &= \lfloor v \rfloor^\lambda \lambda \\
&= s_{i_1} \cdots s_{i_{k-1}} s_{i_k}s_{i_{k+1}} \cdots s_{i_l} \lambda \\
&= s_{i_1} \cdots s_{i_{k-1}}s_{i_{k+1}} \cdots s_{i_l} \lambda,
\end{align*}
and hence $s_{i_1} \cdots s_{i_{k-1}}s_{i_{k+1}} \cdots s_{i_l} \in vW_\lambda$. However, this contradicts the fact that $\lfloor v \rfloor^\lambda$ is minimal in $vW_\lambda$ since $\ell(s_{i_1} \cdots s_{i_{k-1}}s_{i_{k+1}} \cdots s_{i_l}) \leq l-1 < l = \ell(\lfloor v \rfloor^\lambda)$.
Thus, $\langle s_{i_{k+1}} \cdots s_{i_l} \lambda, \alpha_{i_k}^\vee \rangle >0$. This implies that the height function $h_{i_k}^{\pi^{s_{i_{k+1}} \cdots s_{i_l} \lambda}}$ is strictly increasing on $[0, 1]$. Hence $\tilde{f}_{i_k}(\pi^{s_{i_{k+1}} \cdots s_{i_l} \lambda}) \not= 0$. Next we show that $\tilde{f}_{i_k}(\pi^{s_{i_{k+1}} \cdots s_{i_l} \lambda}) \in \mathcal{B}_v(\lambda)$. To prove this, it is sufficient to see that the initial direction of $\tilde{f}_{i_k}(\pi^{s_{i_{k+1}} \cdots s_{i_l} \lambda})$ is less than or equal to $v\lambda$. The initial direction of $\tilde{f}_{i_k}(\pi^{s_{i_{k+1}} \cdots s_{i_l} \lambda})$ is equal to $s_{i_k}s_{i_{k+1}} \cdots s_{i_l} \lambda$. Hence we need to show that $s_{i_k}s_{i_{k+1}} \cdots s_{i_l} \lambda \leq v\lambda$. By the same proof as that of $\langle s_{i_{k+1}} \cdots s_{i_l} \lambda, \alpha_{i_k}^\vee \rangle >0$ above, we can show that $\langle s_{i_{r+1}} \cdots s_{i_l} \lambda, \alpha_{i_r}^\vee \rangle >0$ for all $r \in \{1, \ldots, k\}$. From these inequalities, we deduce that $s_{i_{k+1}} \cdots s_{i_l} \lambda < s_{i_1} \cdots s_{i_k}s_{i_{k+1}} \cdots s_{i_l} \lambda = \lfloor v \rfloor^\lambda \lambda = v\lambda$, as desired.

Suppose, for a contradiction, that $C$ is isomorphic to some Demazure crystal. Recall that $\tilde{f}_{i_k} (\pi^{s_{i_{k+1}} \cdots s_{i_l} \lambda}) \not= 0$ and that $\tilde{e}_{i_k}(\pi^{w\mu}) = 0$. Therefore, $\varphi_{i_k} (\pi^{s_{i_{k+1}} \cdots s_{i_l} \lambda}) > 0 = \varepsilon_{i_k} (\pi^{w\mu})$, and hence we have $\tilde{f}_{i_k}(\pi^{s_{i_{k+1}} \cdots s_{i_l} \lambda} \otimes \pi^{w\mu}) = \tilde{f}_{i_k}(\pi^{s_{i_{k+1}} \cdots s_{i_l} \lambda}) \otimes \pi^{w\mu} \in \mathcal{B}_v(\lambda) \otimes \mathcal{B}_w(\mu)$. This implies that $\tilde{f}_{i_k}(\pi^{s_{i_{k+1}} \cdots s_{i_l} \lambda} \otimes \pi^{w\mu}) \in C$. Let $n = \varphi_{i_k}(\pi^{s_{i_{k+1}} \cdots s_{i_l} \lambda})$.
By the tensor product rule, we obtain that
\begin{align*}
\tilde{f}_{i_k}^{n+1}(\pi^{s_{i_{k+1}} \cdots s_{i_l} \lambda} \otimes \pi^{w\mu}) = \tilde{f}_{i_k}^{n}(\pi^{s_{i_{k+1}} \cdots s_{i_l} \lambda}) \otimes \tilde{f}_{i_k}(\pi^{w\mu}) \not= 0.
\end{align*}
However, $\tilde{f}_{i_k}(\pi^{w\mu}) \not\in \mathcal{B}_w(\mu)$, and hence $\tilde{f}_{i_k}^{n+1}(\pi^{s_{i_{k+1}} \cdots s_{i_l} \lambda} \otimes \pi^{w\mu}) \not\in \mathcal{B}_v(\lambda) \otimes \mathcal{B}_w(\mu)$. Thus $\tilde{f}_{i_k}^{n+1}(\pi^{s_{i_{k+1}} \cdots s_{i_l} \lambda} \otimes \pi^{w\mu}) \not\in C$. This contradicts Lemma \ref{string}. Therefore, $C$ is not isomorphic to any Demazure crystal.
\end{proof}

\section{The recursive formula describing connected components}\label{recursion}

In this section, we consider a recursive formula describing connected components of tensor products of Demazure crystals. We prove the following theorem in this section.
\begin{thm}\label{main_rec}
Let $\lambda, \mu$ be dominant integral weights, and $v, w$ elements of the Weyl group $W$. Take $i \in I$ which satisfies $\ell (s_i v) > \ell (v)$. Also, we take $\pi \in \mathcal{B}_w(\mu)^\lambda$. Then we have the following formula.
\begin{align*}
& \ C(\pi, s_i v) \\
=& \ \left( \bigcup_{a \geq 0} \tilde{f}_i^a(C(\pi, v)) \setminus \{ 0 \} \right) \\
& \ \setminus \left\{ \tilde{f}_{i}^{\langle \wt(\pi_1), \alpha_i^\vee \rangle} (\pi_1) \otimes \tilde{f}_{i}^b(\pi_2) \left| \begin{array}{l}\pi_1 \otimes \pi_2 \in C(\pi, v), \\ \tilde{e}_{i}(\pi_1) = 0, \ \tilde{f}_{i}(\pi_2) \not\in \mathcal{B}_w(\mu) \sqcup \{ 0 \}, \\ 1 \leq b \leq \langle \wt(\pi_2), \alpha_i^\vee \rangle. \end{array} \right. \right\}.
\end{align*}
\end{thm}

From now on, we fix $\lambda, \mu \in P^+$, and $w \in W$. Also, we fix a reduced expression $v = s_{i_1} \cdots s_{i_l}$ for $v$, and set $\text{\boldmath $i$} := (i_1, \ldots, i_l)$. First, we introduce some notation.
\begin{defn}
Let $\pi \in \mathcal{B}_w(\mu)^\lambda$. We take $i \in I$ such that $\ell(s_i v) > \ell(v)$. We set
\begin{align*}
& \ E(\pi, v, i)\\
 := \ & \left\{ \tilde{f}_{i}^{\langle \wt(\pi_1), \alpha_i^\vee \rangle} (\pi_1) \otimes \tilde{f}_{i}^b(\pi_2) \left| \begin{array}{l}\pi_1 \otimes \pi_2 \in C(\pi, v), \\ \tilde{e}_{i}(\pi_1) = 0, \ \tilde{f}_{i}(\pi_2) \not\in \mathcal{B}_w(\mu) \sqcup \{ 0 \}, \\ 1 \leq b \leq \langle \wt(\pi_2), \alpha_i^\vee \rangle. \end{array} \right. \right\}.
\end{align*}
\end{defn}

As the first step of the proof of Theorem \ref{main_rec}, we consider the string decomposition of $C(\pi, v)$, $C(\pi, s_i v)$, $D(\pi, \text{\boldmath $i$})$, and $D(\pi, (i, \text{\boldmath $i$}))$, where $(i, \text{\boldmath $i$}) := (i, i_1, \dots, i_l)$.

We take $\pi_1 \otimes \pi_2 \in C(\pi, v)$ such that $\tilde{e}_i(\pi_1 \otimes \pi_2) = 0$, and set 
\begin{align*}
S &:= \{ \tilde{f}_i^a(\pi_1 \otimes \pi_2) \ | \ a \geq 0 \} \cap C(\pi, v), \\
\widetilde{S} &:= \{ \tilde{f}_i^a(\pi_1 \otimes \pi_2) \ | \ a \geq 0 \} \cap D(\pi, \text{\boldmath $i$}), \\
S[i] &:= \{ \tilde{f}_i^a(\pi_1 \otimes \pi_2) \ | \ a \geq 0 \} \cap C(\pi, s_i v), \\
\widetilde{S}[i] &:= \{ \tilde{f}_i^a(\pi_1 \otimes \pi_2) \ | \ a \geq 0 \} \setminus \{ 0 \} \subset D(\pi, (i, \text{\boldmath $i$})).
\end{align*}
Note that $\widetilde{S}[i] = \{ \tilde{f}_i^a(\pi_1 \otimes \pi_2) \ | \ a \geq 0 \} \cap D(\pi, (i, \text{\boldmath $i$}))$.

By the string property for $D(\pi, \text{\boldmath $i$})$ and $D(\pi, (i, \text{\boldmath $i$}))$, for the relations among $S$, $\widetilde{S}$, $S[i]$, and $\widetilde{S}[i]$, there are only the following four cases:
\begin{enumerate}[{(}1{)}]
\item $S = \widetilde{S} = \widetilde{S}[i]$;
\item $S = \{ \pi_1 \otimes \pi_2 \} \subsetneq \widetilde{S} = \widetilde{S}[i]$;
\item $\{ \pi_1 \otimes \pi_2 \} \subsetneq S \subsetneq \widetilde{S} = \widetilde{S}[i]$;
\item $S = \widetilde{S} = \{ \pi_1 \otimes \pi_2 \}$.
\end{enumerate}

\begin{lem}
If $S \subsetneq S[i]$, then $S = \widetilde{S} = \{ \pi_1 \otimes \pi_2 \}$.
\end{lem}

\begin{proof}
First, observe that if $S = \widetilde{S} = \widetilde{S}[i]$, then $S = S[i]$ since $ S[i] \subset \bigcup_{n \geq 0} \tilde{f}_i^n(S) \setminus \{ 0 \} = \widetilde{S}[i] = S \subset S[i]$.

Now, suppose, for a contradiction, that $S \subsetneq \widetilde{S}$; then, we are in the case (2) or (3) above. In this case, there exists an element $\pi_a \otimes \pi_b \in \widetilde{S} \setminus S$. Since $\widetilde{S} \subset \mathcal{B}_{\text{\boldmath $i$}, \lambda, w\mu} = \bigcup_{a_1, \ldots, a_r \geq 0} \tilde{f}_{i_1}^{a_1} \cdots \tilde{f}_{i_r}^{a_r} (\mathcal{B}_e(\lambda) \otimes \mathcal{B}_w(\mu) ) \setminus \{ 0 \}$, the path $\pi_a$ is contained in $\mathcal{B}_v(\lambda) = \{ \tilde{f}_{i_1}^{a_1} \cdots \tilde{f}_{i_r}^{a_r} (\pi^\lambda) \ | \ a_1, \ldots, a_r \geq 0 \} \setminus \{ 0 \}$. We assume that $\pi_a \otimes \pi_b \in S[i]$. Then the element $\pi_a \otimes \pi_b$ is contained in $\mathcal{B}_{s_i v}(\lambda) \otimes \mathcal{B}_w(\mu)$ since $S[i]$ is the subset of $\mathcal{B}_{s_i v}(\lambda) \otimes \mathcal{B}_w(\mu)$. Hence the path $\pi_b$ is contained in $\mathcal{B}_w(\mu)$. Therefore, $\pi_a \otimes \pi_b \in \mathcal{B}_v(\lambda) \otimes \mathcal{B}_w(\mu)$ which implies that $\pi_a \otimes \pi_b \in C(\pi, v)$. This contradicts the assumption that $\pi_a \otimes \pi_b \not\in S$. Thus $\pi_a \otimes \pi_b \not\in S[i]$, and hence $\widetilde{S} \setminus S \subset \widetilde{S} \setminus S[i]$. Therefore, $S[i] \subset S$, and hence $S = S[i]$.
\end{proof}

As a corollary, we obtain the following inclusion relation.

\begin{cor}
The following inclusion holds.
\begin{align*}
C(\pi, s_i v) \subset C(\pi, v) \cup \left\{  \tilde{f}_i^a (\pi_1 \otimes \pi_2 ) \left| \begin{array}{l}\pi_1 \otimes \pi_2 \in C(\pi, v), \\ \tilde{e}_i(\pi_1 \otimes \pi_2) = 0, \\ \tilde{f}_i^a(\pi_1 \otimes \pi_2) \not\in D(\pi, \text{\boldmath $i$})\sqcup \{ 0 \}, \\ a \geq 0. \end{array} \right. \right\} \setminus \{ 0 \}.
\end{align*}
\end{cor}

By the tensor product rule, we can simplify one condition in the above.

\begin{lem}\label{rec_lem1}
Let $\pi_1 \otimes \pi_2 \in \mathcal{B}(\lambda) \otimes \mathcal{B}(\mu)$. Then the following statements are equivalent:

\begin{enumerate}[{(}1{)}]
\item $\tilde{e}_i(\pi_1 \otimes \pi_2) = 0$;
\item $\tilde{e}_i(\pi_1) = 0$ and $\tilde{e}_i^{\langle \wt(\pi_1), \alpha_i^\vee \rangle + 1} (\pi_2) = 0$.
\end{enumerate}
\end{lem}

\begin{proof}
This is because the following are equivalent:
\begin{align*}
& \tilde{e}_i(\pi_1 \otimes \pi_2) = 0 \\
\Longleftrightarrow \ & \varphi_i(\pi_1) \geq \varepsilon_i(\pi_2) \text{ and } \tilde{e}_i(\pi_1) = 0 \\
\Longleftrightarrow \ & \tilde{e}_i(\pi_1) = 0 \text{ and } \varepsilon_i(\pi_2) \leq \langle \wt(\pi_1), \alpha_i^\vee \rangle \\
\Longleftrightarrow \ & \tilde{e}_i(\pi_1) = 0 \text{ and } \tilde{e}_i^{\langle \wt(\pi_1), \alpha_i^\vee \rangle + 1}(\pi_2) = 0.
\end{align*}
\end{proof}

\begin{cor}\label{rec_step1}
The following inclusion holds.
\begin{align*}
& \ C(\pi, s_i v) \\ \subset & \ C(\pi, v) \cup \left\{  \tilde{f}_i^a (\pi_1 \otimes \pi_2 ) \left| \begin{array}{l}\pi_1 \otimes \pi_2 \in C(\pi, v), \\ \tilde{e}_i(\pi_1) = 0, \ \tilde{e}_i^{\langle \wt(\pi_1), \alpha_i^\vee \rangle + 1} (\pi_2) = 0, \\ \tilde{f}_i(\pi_1 \otimes \pi_2) \not\in D(\pi, \text{\boldmath $i$})\sqcup \{ 0 \}, \\ a \geq 0. \end{array} \right. \right\} \setminus \{ 0 \}.
\end{align*}
\end{cor}

\begin{defn}
We set
\begin{align*}
\varDelta^\prime (\pi, v, \text{\boldmath $i$}, i) := \left\{ \tilde{f}_{i}^{a} (\pi_1 \otimes \pi_2) \ \left| \begin{array}{l}\pi_1 \otimes \pi_2 \in C(\pi, v), \\ \tilde{e}_{i}(\pi_1) = 0 , \ \tilde{e}_{i}^{\langle \wt (\pi_1), \alpha_i^\vee \rangle + 1} (\pi_2) = 0, \\ \tilde{f}_i(\pi_1 \otimes \pi_2) \not\in D(\pi, \text{\boldmath $i$}) \sqcup \{ 0 \}, \\ a \geq 0.\end{array} \right. \right\} \setminus \{ 0 \}.
\end{align*}
\end{defn}

As the second step of the proof of Theorem \ref{main_rec}, we remove some unnecessary conditions from the right-hand side of the inclusion in Corollary \ref{rec_step1}.

\begin{prop}\label{rec_step2}
The following equality holds.
\begin{align*}
& \ C(\pi, s_i v) \\
=& \ (C(\pi, v) \cup \varDelta^\prime (\pi, v, \text{\boldmath $i$}, i)) \\
& \ \setminus \left\{ \tilde{f}_{i}^{\max\{ \varphi_i(\pi_1) - \varepsilon_i(\pi_2), 0 \}} (\pi_1) \otimes \tilde{f}_{i}^b(\pi_2) \left| \begin{array}{l}\pi_1 \otimes \pi_2 \in C(\pi, v), \\ \tilde{e}_{i}(\pi_1 \otimes \pi_2) = 0, \\ \tilde{f}_{i}(\pi_2) \not\in \mathcal{B}_w(\mu) \sqcup \{ 0 \}, \\ 1 \leq b \leq \varphi_i(\pi_2). \end{array} \right. \right\}.
\end{align*}
\end{prop}

\begin{defn}
We set
\begin{align*}
E^\prime (\pi, v, i) := \left\{ \tilde{f}_{i}^{\max\{ \varphi_i(\pi_1) - \varepsilon_i(\pi_2), 0 \}} (\pi_1) \otimes \tilde{f}_{i}^b(\pi_2) \left| \begin{array}{l}\pi_1 \otimes \pi_2 \in C(\pi, v), \\ \tilde{e}_{i}(\pi_1 \otimes \pi_2) = 0, \\ \tilde{f}_{i}(\pi_2) \not\in \mathcal{B}_w(\mu) \sqcup \{ 0 \}, \\ 1 \leq b \leq \varphi_i(\pi_2). \end{array} \right. \right\}.
\end{align*}
\end{defn}

\begin{proof}[Proof of Proposition \ref{rec_step2}]
Since $C(\pi, v) \subset C(\pi, s_i v)$, it is sufficient to find all those elements in $\varDelta^\prime (\pi, v, \text{\boldmath $i$}, i)$ which do not belong to $C(\pi, s_i v)$. To find these elements, we need to find all these $\pi_1 \otimes \pi_2 \in C(\pi, v)$ and $a \geq 0$ which satisfy the following conditions:
\begin{enumerate}[{(}1{)}]
\item $\tilde{e}_i(\pi_1 \otimes \pi_2) = 0$;
\item $\tilde{f}_i(\pi_1 \otimes \pi_2) \not\in D(\pi, \text{\boldmath $i$}) \sqcup \{ 0 \}$;
\item $\tilde{f}_i^a(\pi_1 \otimes \pi_2) \not\in C(\pi, s_i v) \sqcup \{ 0 \}$.
\end{enumerate}

Take an element $\pi_1 \otimes \pi_2 \in C(\pi, v)$ which satisfies conditions (1) and (2) above. We set $M := \max\{ \varphi_i(\pi_1) - \varepsilon_i(\pi_2), 0 \}$.

If $a$ satisfies $0 \leq a \leq M$, then $\tilde{f}_i^a(\pi_1 \otimes \pi_2)$ is equal to $\tilde{f}_i^a(\pi_1) \otimes \pi_2$ by the tensor product rule. Hence $\tilde{f}_i^a(\pi_1 \otimes \pi_2) \in C(\pi, s_i v)$ since $\tilde{f}_i^a(\pi_1) \in \mathcal{B}_{s_i v}(\lambda)$. If $a$ satisfies $a > M + \varphi_i(\pi_2)$, then
\begin{align*}
\tilde{f}_i^a(\pi_1 \otimes \pi_2) &= \tilde{f}_i^M (\pi_1) \otimes \tilde{f}_i^{a-M} (\pi_2) \\
&= \tilde{f}_i^M (\pi_1) \otimes 0 \\
&= 0
\end{align*}
since $a-M > \varphi_i(\pi_2)$.

From now on, we assume that $M < a \leq M + \varphi_i(\pi_2)$. First we assume that $\tilde{f}_i(\pi_2) \in \mathcal{B}_w(\mu) \sqcup \{ 0 \}$. Then $\tilde{f}_i^{a-M}(\pi_2) \in \mathcal{B}_w(\mu)$ by the string property for $\mathcal{B}_w(\mu)$. Therefore, we have
\begin{align*}
\tilde{f}_i^a(\pi_1 \otimes \pi_2) = \tilde{f}_i^M(\pi_1) \otimes \tilde{f}_i^{a-M}(\pi_2) \in \mathcal{B}_{s_i v}(\lambda) \otimes \mathcal{B}_w(\mu).
\end{align*}
Hence $\tilde{f}_i^a(\pi_1 \otimes \pi_2) \in C(\pi, s_i v)$. Next we assume that $\tilde{f}_i(\pi_2) \not\in \mathcal{B}_w(\mu) \sqcup \{ 0 \}$. Then $\tilde{f}_i^{a-M} (\pi_2)$ does not belong to $\mathcal{B}_w(\mu) \sqcup \{ 0 \}$. Therefore,
\begin{align*}
\tilde{f}_i^a(\pi_1 \otimes \pi_2) = \tilde{f}_i^M(\pi_1) \otimes \tilde{f}_i^{a-M} (\pi_2) \not\in \mathcal{B}_{s_i v}(\lambda) \otimes \mathcal{B}_w(\mu) \sqcup \{ 0 \}.
\end{align*}
Hence $\tilde{f}_i^a(\pi_1 \otimes \pi_2) \not\in C(\pi, s_i v) \sqcup \{ 0 \}$.

From the above discussion, we conclude that the set of all those elements in $\varDelta^\prime (\pi, v, \text{\boldmath $i$}, i)$ which do not belong to $C(\pi, s_i v)$, is identical to $E^\prime(\pi, v, i)$.
\end{proof}

Also, we can simplify the conditions of $E^\prime(\pi, v, i)$.

\begin{lem}\label{rec_lem2}
Take a path $\pi \in \mathcal{B}_w(\mu)$ which satisfies $\tilde{f}_i(\pi) \not\in \mathcal{B}_w(\mu) \sqcup \{ 0 \}$. Then, $\varepsilon_i(\pi) = 0$ and $\varphi_i(\pi) = \langle \wt(\pi), \alpha_i^\vee \rangle$.
\end{lem}

\begin{proof}
By the string property for $\mathcal{B}_w(\mu)$ (see Lemma \ref{string}), the path $\pi$ is an $i$-highest element, that is, $\tilde{e}_i(\pi) = 0$, which implies that $\varepsilon_i(\pi) = 0$. Hence, by an axiom for crystals, we have $\varphi_i(\pi) = \varepsilon_i(\pi) + \langle \wt(\pi), \alpha_i^\vee \rangle = \langle \wt(\pi), \alpha_i^\vee \rangle$.
\end{proof}

\begin{lem}
Let $\pi_1 \otimes \pi_2 \in \mathcal{B}_v(\lambda) \otimes \mathcal{B}_w(\mu)$. If $\tilde{f}_i(\pi_2) \not\in \mathcal{B}_w(\mu) \sqcup \{ 0 \}$, then the following conditions are equivalent:
\begin{enumerate}[{(}1{)}]
\item $\tilde{e}_i(\pi_1 \otimes \pi_2) = 0$;
\item $\tilde{e}_i(\pi_1) = 0$.
\end{enumerate}
\end{lem}

\begin{proof}
By Lemma \ref{rec_lem2}, we have $\tilde{e}_i(\pi_2) = 0$, and hence $\tilde{e}_i^{\langle \wt(\pi_1), \alpha_i^\vee \rangle + 1} (\pi_2) = 0$. Therefore, this lemma follows from Lemma \ref{rec_lem1}.
\end{proof}

\begin{cor}\label{rec_step2_cor}
It holds that
\begin{align*}
C(\pi, s_i v) = (C(\pi, v) \cup \varDelta^\prime (\pi, v, \text{\boldmath $i$}, i)) \setminus E(\pi, v, i).
\end{align*}
\end{cor}

\begin{thm}\label{rec_step3}
Let $\lambda, \mu$ be dominant integral weights, and $v, w$ elements of the Weyl group $W$. Take $i \in I$ which satisfies $\ell (s_i v) > \ell (v)$. Also, we take $\pi \in \mathcal{B}_w(\mu)^\lambda$. Then we have the following formula.
\begin{align*}
& \ C(\pi, s_i v) \\
=& \ \left( \left\{ \tilde{f}_{i}^{a} (\pi_1 \otimes \pi_2) \ \left| \begin{array}{l}\pi_1 \otimes \pi_2 \in C(\pi, v), \\ \tilde{e}_{i}(\pi_1) = 0 , \ \tilde{e}_{i}^{\langle \wt (\pi_1), \alpha_i^\vee \rangle + 1} (\pi_2) = 0, \\ a \geq 0.\end{array} \right. \right\} \setminus \{ 0 \} \right)\\
& \ \setminus \left\{ \tilde{f}_{i}^{\langle \wt(\pi_1), \alpha_i^\vee \rangle} (\pi_1) \otimes \tilde{f}_{i}^b(\pi_2) \left| \begin{array}{l}\pi_1 \otimes \pi_2 \in C(\pi, v), \\ \tilde{e}_{i}(\pi_1) = 0, \ \tilde{f}_{i}(\pi_2) \not\in \mathcal{B}_w(\mu) \sqcup \{ 0 \}, \\ 1 \leq b \leq \langle \wt(\pi_2), \alpha_i^\vee \rangle. \end{array} \right. \right\}.
\end{align*}
\end{thm}

\begin{defn}
Let $\pi \in \mathcal{B}_w(\mu)^\lambda$. We take $i \in I$ such that $\ell(s_i v) > \ell(v)$. We set 
\begin{align*}
\widetilde{C}(\pi, v, i) := \left\{ \tilde{f}_{i}^{a} (\pi_1 \otimes \pi_2) \ \left| \begin{array}{l}\pi_1 \otimes \pi_2 \in C(\pi, v), \\ \tilde{e}_{i}(\pi_1) = 0 , \ \tilde{e}_{i}^{\langle \wt (\pi_1), \alpha_i^\vee \rangle + 1} (\pi_2) = 0, \\ a \geq 0.\end{array} \right. \right\} \setminus \{ 0 \}.
\end{align*}
\end{defn}

\begin{proof}[Proof of Theorem \ref{rec_step3}]
By Corollary \ref{rec_step2_cor} and the fact that $C(\pi, v)$ is included in $\widetilde{C}(\pi, v, i)$, it is sufficient to prove that
\begin{align*}
\widetilde{C} (\pi, v, i) \setminus (C(\pi, v) \cup \varDelta^\prime(\pi, v, \text{\boldmath $i$}, i)) \subset E(\pi, v, i).
\end{align*}

Take an element $\pi_1 \otimes \pi_2 \in C(\pi, v)$ satisfying the following conditions:
\begin{enumerate}[{(}1{)}]
\item $\tilde{e}_i(\pi_1) = 0$;
\item $\tilde{e}_i^{\langle \wt (\pi_1) , \alpha_i^\vee \rangle + 1} (\pi_2) = 0$;
\item $\tilde{f}_i(\pi_1 \otimes \pi_2) \in D(\pi, \text{\boldmath $i$})$.
\end{enumerate}

Assume that there exists $a \geq 0$ such that $\tilde{f}_i^a(\pi_1 \otimes \pi_2) \not\in C(\pi, v) \sqcup \{ 0 \}$. We need to prove that $\tilde{f}_i^a(\pi_1 \otimes \pi_2) \in E(\pi, v, i)$.

First, we show that $\tilde{f}_i(\pi_2) \not\in \mathcal{B}_w(\mu) \sqcup \{ 0 \}$. Suppose that $\tilde{f}_i(\pi_2) \in \mathcal{B}_w(\mu) \sqcup \{ 0 \}$. Since $\tilde{f}_i(\pi_1 \otimes \pi_2) \in D(\pi, \text{\boldmath $i$})$ and $D(\pi, \text{\boldmath $i$})$ is a subset of $\mathcal{B}_{\text{\boldmath $i$}, \lambda, w\mu}$, the path $\tilde{f}_i^c(\pi_1)$ belongs to $\mathcal{B}_v(\lambda)$ for $c$ such that $0 \leq c \leq \max\{\varphi_i(\pi_1) - \varepsilon_i(\pi_2), 0\}$. Therefore, $\tilde{f}_i^a(\pi_1 \otimes \pi_2) \in \mathcal{B}_v(\lambda) \otimes \mathcal{B}_w(\mu)$ which shows that $\tilde{f}_i^a(\pi_1 \otimes \pi_2) \in C(\pi, v)$, a contradiction. Thus, $\tilde{f}_i(\pi_2) \not\in \mathcal{B}_w(\mu) \sqcup \{ 0 \}$.

Since $\tilde{e}_i(\pi_1) = 0$ or equivalently $\varepsilon_i(\pi_1) = 0$, we have $\varphi_i(\pi_1) = \langle \wt (\pi_1), \alpha_i^\vee \rangle$ by an axiom of crystals. Also, we have $\varepsilon_i(\pi_2) = 0$ by Lemma \ref{rec_lem2}. Hence $\max\{ \varphi_i(\pi_1) - \varepsilon_i(\pi_2), 0 \} = \varphi_i(\pi_1) = \langle \wt (\pi_1), \alpha_i^\vee \rangle$.

Now, suppose that $0 \leq a \leq \langle \wt (\pi_1), \alpha_i^\vee \rangle$. Then $\tilde{f}_i^a (\pi_1 \otimes \pi_2) = \tilde{f}_i^a(\pi_1) \otimes \pi_2$. However, we have already shown that $\tilde{f}_i^a(\pi_1) \in \mathcal{B}_v(\lambda)$, and $\pi_2 \in \mathcal{B}_w(\mu)$ by the definition. Hence it follows that $\tilde{f}_i^a (\pi_1 \otimes \pi_2) \in \mathcal{B}_v(\lambda) \otimes \mathcal{B}_w(\mu)$. This implies that $\tilde{f}_i^a (\pi_1 \otimes \pi_2) \in C(\pi, v)$, a contradiction. Thus, $a > \langle \wt (\pi_1), \alpha_i^\vee \rangle$. Set $b := a - \langle \wt (\pi_1), \alpha_i^\vee \rangle$. Then $b \geq 1$ and $\tilde{f}_i^a (\pi_1 \otimes \pi_2) = \tilde{f}_i^{\langle \wt (\pi_1), \alpha_i^\vee \rangle}(\pi_1) \otimes \tilde{f}_i^b(\pi_2)$. Since $\tilde{f}_i^a (\pi_1 \otimes \pi_2) \not= 0$, we have $\tilde{f}_i^b(\pi_2) \not= 0$, which implies that $b \leq \varphi_i(\pi_2)$.

From the above argument, we can verify that $\tilde{f}_i^a(\pi_1 \otimes \pi_2) \in E(\pi, v, i)$.
\end{proof}

\begin{proof}[Proof of Theorem \ref{main_rec}]
It remains to prove
\begin{align}\label{rest_proof}
\bigcup_{a \geq 0} \tilde{f}_i^a(C(\pi, v)) \setminus \{ 0 \} = \widetilde{C}(\pi, v, i).
\end{align}

First, take an element $\xi \in \bigcup_{a \geq 0} \tilde{f}_i^a(C(\pi, v)) \setminus \{ 0 \}$. Then, there exist $a \geq 0$ and $\pi \otimes \pi^\prime \in C(\pi, v)$ such that $\xi = \tilde{f}_i^a(\pi \otimes \pi^\prime)$. Take $n \geq 0$ such that $\tilde{e}_i^{n}(\pi \otimes \pi^\prime) \not= 0$ and that $\tilde{e}_i^{n+1}(\pi \otimes \pi^\prime) = 0$. Set $\pi_1 \otimes \pi_2 := \tilde{e}_i^{n}(\pi \otimes \pi^\prime)$. Then, $\pi_1 \otimes \pi_2$ satisfies the following conditions:
\begin{enumerate}[{(}1{)}]
\item $\pi_1 \otimes \pi_2 \in C(\pi, v)$;
\item $\tilde{e}_i(\pi_1) = 0$, $\tilde{e}_i^{\langle \wt(\pi_1) , \alpha_i^\vee \rangle + 1}(\pi_2) = 0$;
\item $\pi \otimes \pi^\prime = \tilde{f}_i^n(\pi_1 \otimes \pi_2)$.
\end{enumerate}
Hence $\xi = \tilde{f}_i^{a+n}(\pi_1 \otimes \pi_2) \in \widetilde{C}(\pi, v, i)$, which implies that
\begin{align*}
\bigcup_{a \geq 0} \tilde{f}_i^a(C(\pi, v)) \setminus \{ 0 \} \subset \widetilde{C}(\pi, v, i).
\end{align*}

On the other hand, by the definition of $\widetilde{C}(\pi, v, i)$, it is clear that
\begin{align*}
\bigcup_{a \geq 0} \tilde{f}_i^a(C(\pi, v)) \setminus \{ 0 \} \supset \widetilde{C}(\pi, v, i).
\end{align*}

From the above argument, the proof of (\ref{rest_proof}) is completed, and hence Theorem \ref{main_rec} follows.
\end{proof}

\section{Some applications}\label{application}

In this section, we give some applications of the theorems that we proved above.

\subsection{Key positivity problem}
For $\nu \in P$, there is a unique $\lambda_\nu \in W\nu$ such that $\lambda_\nu \in P^+$. In this situation, there exists $u \in W$ such that $u\lambda_\nu = \nu$. Take $u_\nu := \lfloor u \rfloor^{\lambda_\nu}$, and set $\kappa_\nu := \ch(\mathcal{B}_{u_\nu}(\lambda_\nu))$. Note that $u_\nu$ does not depend on the choice of $u$.

If $\mathfrak{g} = \mathfrak{sl}_{n+1}$, then $\kappa_\nu$ is identical to a {\it key polynomial}. For details about key polynomials, see \cite{Reiner}. It is known that a product of key polynomials is a linear combination of key polynomials with integer coefficients.
\begin{thm}[{\cite[Corollary 7]{Reiner}}]
Assume that $\mathfrak{g} = \mathfrak{sl}_{n+1}$. Let $\lambda$ and $\mu$ be dominant integral weights, and $v, w$ elements of $W$. Then there exists an integer $a_{v, w, \lambda, \mu}^{\nu}$ for each $\nu \in P$ such that 
\begin{align*}
\kappa_{v\lambda} \kappa_{w\mu} = \sum_{\nu \in P} a_{v, w, \lambda, \mu}^{\nu} \kappa_{\nu}.
\end{align*}
\end{thm}

The key positivity problem is the problem whether integers $a_{v, w, \lambda, \mu}^{\nu}$ are nonnegative or not. In some special cases, we have already known $a_{v, w, \lambda, \mu}^{\nu} \geq 0$ for each $\nu \in P$.

\begin{thm}\label{keypos_known1}
Let $\lambda, \mu \in P^+$, and $w \in W$. Then there exists a nonnegative integer $a_{e, w, \lambda, \mu}^{\nu}$ for each $\nu \in P$ such that 
\begin{align*}
\kappa_{\lambda} \kappa_{w\mu} = \sum_{\nu \in P} a_{e, w, \lambda, \mu}^{\nu} \kappa_{\nu}.
\end{align*}
\end{thm}

\begin{proof}
Take characters of both sides of the isomorphism in Theorem \ref{1tensDem}.
\end{proof}

\begin{rem}
For $\lambda \in P^+$, we have $\mathcal{B}_e(\lambda) = \{ b_\lambda \}$. Hence $\kappa_\lambda = \ch(\mathcal{B}_e(\lambda)) = e^\lambda$.
\end{rem}

\begin{thm}[{\cite[Theorem 6.1]{Haglund}}]\label{refLRrule}
Assume that $\mathfrak{g} = \mathfrak{sl}_{n+1}$. Let $\lambda, \mu \in P^+$, and $v \in W$. Then there exists a nonnegative integer $a_{v, w_\circ, \lambda, \mu}^{\nu}$ for each $\nu \in P$ such that 
\begin{align*}
\kappa_{v\lambda} \kappa_{w_\circ \mu} = \sum_{\nu \in P} a_{v, w_\circ, \lambda, \mu}^{\nu} \kappa_{\nu}.
\end{align*}
\end{thm}

\begin{rem}
(1) For $\mu \in P^+$, we have $\kappa_{w_\circ \mu} = \ch(\mathcal{B}_{w_\circ}(\mu)) = \ch(\mathcal{B}(\mu))$. Hence if $\mathfrak{g} = \mathfrak{sl}_{n+1}$, then $\kappa_{w_\circ \mu}$ is identical to a Schur polynomial.

\noindent (2) Haglund, Luoto, Maoson, and van Willigenburg wrote an explicit formula for $a_{v, w_\circ, \lambda, \mu}^{\nu}$ in terms of Littlewood-Richardson key skylines in \cite{Haglund}.
\end{rem}

However, the key positivity problem is an open problem. We obtain a partial solution for the key positivity problem, which is a refinement of Theorem \ref{keypos_known1} and Theorem \ref{refLRrule}.

\begin{thm}\label{key_pos}
Let $\lambda$ and $\mu$ be dominant integral weights, and $v, w$ elements of $W$. If $\lfloor v \rfloor^\lambda \in W_{\lceil w \rceil^\mu}$, then there exists a nonnegative integer $a_{v, w, \lambda, \mu}^{\nu}$ for each $\nu \in P$ such that 
\begin{align*}
\kappa_{v\lambda} \kappa_{w\mu} = \sum_{\nu \in P} a_{v, w, \lambda, \mu}^{\nu} \kappa_{\nu}.
\end{align*}
\end{thm}

\begin{proof}
By Corollary \ref{maincor}, we have
\begin{align*}
\mathcal{B}_v(\lambda) \otimes \mathcal{B}_w(\mu) &\simeq \bigsqcup_{\pi \in \mathcal{B}_w(\mu)^\lambda} \mathcal{B}_{u(\pi, v)}(\lambda + \wt (\pi)) \\
&= \bigsqcup_{\nu \in P} \mathcal{B}_{u_\nu}(\lambda_\nu)^{\oplus a_{v, w, \lambda, \mu}^{\nu}},
\end{align*}
where $a_{v, w, \lambda, \mu}^{\nu} = \# \{ \pi \in \mathcal{B}_w(\mu)^\lambda \ | \ \wt(\pi) = \lambda_\nu - \lambda, \ \lfloor u(\pi, v) \rfloor^{\lambda_\nu} = u_\nu \}$ are nonnegative integers. By taking characters, we conclude the desired assertion.
\end{proof}

\begin{cor}\label{keypos2}
Let $\lambda$ and $\mu$ be dominant integral weights, and $v, w$ elements of $W$. If $\lfloor v \rfloor^\lambda \in W_{\lceil w \rceil^\mu}$ or $\lfloor w \rfloor^\mu \in W_{\lceil v \rceil^\lambda}$, then the product of $\kappa_{v\lambda}$ and $\kappa_{w\mu}$ is a linear combination of some key polynomials with nonnegative integer coefficients.
\end{cor}

\begin{rem}
(1) For $\lambda, \mu \in P^+$ and $w \in W$, we have $\lfloor e \rfloor^\lambda = e \in W_{\lceil w \rceil^\mu}$. Hence Corollary \ref{keypos2} is a refinement of Theorem \ref{keypos_known1}.

\noindent (2) For $\lambda, \mu \in P^+$ and $v \in W$, we have $\lfloor v \rfloor^\lambda \in W = W_{w_\circ} = W_{\lceil w_\circ \rceil^\mu}$. Hence Corollary \ref{keypos2} is a refinement of Theorem \ref{refLRrule}.
\end{rem}

\begin{rem}
Classically, we know the Schubert positivity property; a product of Schubert polynomials is a linear combination of Schubert polynomials with nonnegative integer coefficients (see \cite[Section 10.6, Exercise 12]{Fulton}). Since a Schubert polynomial for a vexillary permutation is identical to certain key polynomial, the Schubert positivity property shows that key positivity property holds under certain condition. This condition is neither a necessary condition nor a sufficient condition for the condition in Corollary \ref{keypos2}.
\end{rem}

\subsection{The Leibniz rule for root operators}

First, we define maps $S_i$, $i \in I$, which appear in our theorem below.

\begin{defn}
For $i \in I$ and $\lambda \in P^+$, we define
\begin{align*}
S_i : \mathcal{B}(\lambda) \to \mathcal{B}(\lambda), \ b \mapsto \left\{ \begin{array}{ll} \tilde{f}_i^{\langle \wt(b), \alpha_i^\vee \rangle}(b) & \text{if } \langle \wt(b), \alpha_i^\vee \rangle \geq 0, \\ \tilde{e}_i^{-\langle \wt(b), \alpha_i^\vee \rangle}(b) & \text{if } \langle \wt(b), \alpha_i^\vee \rangle \leq 0. \end{array} \right.
\end{align*}

\begin{rem}
First, the following diagram is commutative, where the map $s_i : P \to P$ is the simple reflection corresponding to the simple root $\alpha_i$.
\begin{center}
\begin{tikzcd}
  \mathcal{B}(\lambda) \ar[r, "S_i"] \arrow[d, "\wt"'] & \mathcal{B}(\lambda) \ar[d, "\wt"] \\
  P \ar[r, "s_i"] & P.
\end{tikzcd}
\end{center}
Next, maps $S_i$, $i \in I$, satisfy the braid relations: Take $i, j \in I$. If $s_i s_j \in W$ has order $m$, then
\begin{align*}
S_i \circ S_j \circ S_i \circ S_j \circ \cdots = S_j \circ S_i \circ S_j \circ S_i \circ \cdots, 
\end{align*}
where the number of terms on both sides is $m$.

For these reasons, maps $S_i$, $i \in I$, can be thought of ``lifts of the simple reflections.'' For details, see \cite[Section 8]{Littelmann_path}.
\end{rem}
\end{defn}

Now, we give an equation, which is an analog of the Leibniz rule for Demazure operators.

\begin{thm}
It holds that
\begin{align*}
& \ \bigcup_{n \geq 0} \tilde{f}_{i}^{n} (\mathcal{B}_{v}(\lambda) \otimes \mathcal{B}_{w}(\mu)) \setminus \{ 0 \} \\
 = & \ (\mathcal{B}_{s_i v}(\lambda) \otimes \mathcal{B}_{w}(\mu)) \sqcup (S_{i} (\tilde{e}_i^{\max}(\mathcal{B}_v(\lambda))) \otimes (\mathcal{B}_{s_i w}(\mu) \setminus \mathcal{B}_{w}(\mu))).
\end{align*}
\end{thm}

\begin{proof}
By Theorem \ref{main_rec}, we have
\begin{align*}
& \ \mathcal{B}_{s_i v}(\lambda) \otimes \mathcal{B}_w(\mu) \\
=& \ \bigsqcup_{\pi \in \mathcal{B}_w(\mu)^\lambda} C(\pi, s_i v) \\
=& \ \bigsqcup_{\pi \in \mathcal{B}_w(\mu)^\lambda} \left( \left( \bigcup_{a \geq 0} \tilde{f}_i^a (C(\pi, v)) \setminus \{ 0 \} \right) \setminus E(\pi, v, i) \right) \\
=& \ \bigsqcup_{\pi \in \mathcal{B}_w(\mu)^\lambda} \left( \bigcup_{a \geq 0} \tilde{f}_i^a (C(\pi, v)) \setminus \{ 0 \} \right) \setminus \bigsqcup_{\pi \in \mathcal{B}_w(\mu)^\lambda} E(\pi, v, i) \\
=& \ \left( \bigsqcup_{\pi \in \mathcal{B}_w(\mu)^\lambda} \bigcup_{a \geq 0} \tilde{f}_i^a (C(\pi, v)) \setminus \{ 0 \} \right) \setminus \bigsqcup_{\pi \in \mathcal{B}_w(\mu)^\lambda} E(\pi, v, i) \\
=& \ \left( \bigcup_{a \geq 0} \tilde{f}_i^a \left( \bigsqcup_{\pi \in \mathcal{B}_w(\mu)^\lambda} C(\pi, v) \right) \setminus \{ 0 \} \right) \setminus \bigsqcup_{\pi \in \mathcal{B}_w(\mu)^\lambda} E(\pi, v, i) \\
=& \ \left( \bigcup_{a \geq 0} \tilde{f}_i^a (\mathcal{B}_v(\lambda) \otimes \mathcal{B}_w(\mu)) \setminus \{ 0 \} \right) \setminus \bigsqcup_{\pi \in \mathcal{B}_w(\mu)^\lambda} E(\pi, v, i).
\end{align*}
Hence it is sufficient to show that
\begin{align}
\bigsqcup_{\pi \in \mathcal{B}_w(\mu)^\lambda} E(\pi, v, i) = S_{i} (\tilde{e}_i^{\max}(\mathcal{B}_v(\lambda))) \otimes (\mathcal{B}_{s_i w}(\mu) \setminus \mathcal{B}_{w}(\mu)). \label{to_prove_2}
\end{align}
We have
\begin{align*}
& \ \bigsqcup_{\pi \in \mathcal{B}_w(\mu)^\lambda} E(\pi, v, i)  \\
=& \ \left\{ \tilde{f}_{i}^{\langle \wt(\pi_1), \alpha_i^\vee \rangle} (\pi_1) \otimes \tilde{f}_{i}^b(\pi_2) \left| \begin{array}{l}\pi \in \mathcal{B}_w(\mu)^\lambda, \\ \pi_1 \otimes \pi_2 \in C(\pi, v), \\ \tilde{e}_{i}(\pi_1) = 0, \ \tilde{f}_{i}(\pi_2) \not\in \mathcal{B}_w(\mu) \sqcup \{ 0 \}, \\ 1 \leq b \leq \langle \wt(\pi_2), \alpha_i^\vee \rangle. \end{array} \right. \right\} \\
=& \ \left\{ \tilde{f}_{i}^{\langle \wt(\pi_1), \alpha_i^\vee \rangle} (\pi_1) \otimes \tilde{f}_{i}^b(\pi_2) \left| \begin{array}{l}\pi_1 \otimes \pi_2 \in \bigsqcup_{\pi \in \mathcal{B}_w(\mu)^\lambda} C(\pi, v), \\ \tilde{e}_{i}(\pi_1) = 0, \ \tilde{f}_{i}(\pi_2) \not\in \mathcal{B}_w(\mu) \sqcup \{ 0 \}, \\ 1 \leq b \leq \langle \wt(\pi_2), \alpha_i^\vee \rangle. \end{array} \right. \right\} \\
=& \ \left\{ \tilde{f}_{i}^{\langle \wt(\pi_1), \alpha_i^\vee \rangle} (\pi_1) \otimes \tilde{f}_{i}^b(\pi_2) \left| \begin{array}{l}\pi_1 \otimes \pi_2 \in \mathcal{B}_v(\lambda) \otimes \mathcal{B}_w(\mu), \\ \tilde{e}_{i}(\pi_1) = 0, \ \tilde{f}_{i}(\pi_2) \not\in \mathcal{B}_w(\mu) \sqcup \{ 0 \}, \\ 1 \leq b \leq \langle \wt(\pi_2), \alpha_i^\vee \rangle. \end{array} \right. \right\} \\
=& \ \left\{ \tilde{f}_{i}^{\langle \wt(\pi_1), \alpha_i^\vee \rangle} (\pi_1) \otimes \tilde{f}_{i}^b(\pi_2) \left| \begin{array}{l}\pi_1 \in \mathcal{B}_v(\lambda), \ \pi_2 \in  \mathcal{B}_w(\mu), \\ \tilde{e}_{i}(\pi_1) = 0, \ \tilde{f}_{i}(\pi_2) \not\in \mathcal{B}_w(\mu) \sqcup \{ 0 \}, \\ 1 \leq b \leq \langle \wt(\pi_2), \alpha_i^\vee \rangle. \end{array} \right. \right\} \\
=& \ \left\{ \tilde{f}_{i}^{\langle \wt(\pi_1), \alpha_i^\vee \rangle} (\pi_1) \left| \begin{array}{l}\pi_1 \in \mathcal{B}_v(\lambda), \\ \tilde{e}_{i}(\pi_1) = 0.\end{array} \right. \right\} \otimes \left\{ \tilde{f}_{i}^b(\pi_2) \left| \begin{array}{l}\pi_2 \in  \mathcal{B}_w(\mu), \\ \tilde{f}_{i}(\pi_2) \not\in \mathcal{B}_w(\mu) \sqcup \{ 0 \}, \\ 1 \leq b \leq \langle \wt(\pi_2), \alpha_i^\vee \rangle. \end{array} \right. \right\}.
\end{align*}

Let $\pi_1 \in \mathcal{B}_v(\lambda)$. If $\tilde{e}_i(\pi_1) = 0$, then $\langle \wt (\pi_1), \alpha_1^\vee \rangle \geq 0$. Hence it follows that $\tilde{f}_{i}^{\langle \wt(\pi_1), \alpha_i^\vee \rangle} (\pi_1) = S_i(\pi_1)$. Therefore, we have
\begin{align*}
& \ \left\{ \tilde{f}_{i}^{\langle \wt(\pi_1), \alpha_i^\vee \rangle} (\pi_1) \left| \begin{array}{l}\pi_1 \in \mathcal{B}_v(\lambda), \\ \tilde{e}_{i}(\pi_1) = 0.\end{array} \right. \right\} \\
=& \ \{ \tilde{f}_{i}^{\langle \wt(\pi_1), \alpha_i^\vee \rangle} (\pi_1) \ | \ \pi_1 \in \tilde{e}_i^{\max}(\mathcal{B}_v(\lambda)) \} \\
=& \ \{ S_i(\pi_1) \ | \ \pi_1 \in \tilde{e}_i^{\max}(\mathcal{B}_v(\lambda)) \} \\
=& \ S_i(\tilde{e}_i^{\max}(\mathcal{B}_v(\lambda))).
\end{align*}

Next we prove that
\begin{align}\label{RHS_2}
\left\{ \tilde{f}_{i}^b(\pi_2) \left| \begin{array}{l}\pi_2 \in \mathcal{B}_w(\mu), \\ \tilde{f}_{i}(\pi_2) \not\in \mathcal{B}_w(\mu) \sqcup \{ 0 \}, \\ 1 \leq b \leq \langle \wt(\pi_2), \alpha_i^\vee \rangle. \end{array} \right. \right\} = \mathcal{B}_{s_i w}(\mu) \setminus \mathcal{B}_w(\mu).
\end{align}

Assume that $\ell(s_i w) < \ell(w)$. Then one has $\mathcal{B}_{s_i w}(\mu) \setminus \mathcal{B}_w(\mu) = \emptyset$. Also, there is no $\pi_2 \in \mathcal{B}_w(\mu)$ such that $\tilde{f}_{i}(\pi_2) \not\in \mathcal{B}_w(\mu) \sqcup \{ 0 \}$. Hence we have
\begin{align*}
\left\{ \tilde{f}_{i}^b(\pi_2) \left| \begin{array}{l}\pi_2 \in \mathcal{B}_w(\mu), \\ \tilde{f}_{i}(\pi_2) \not\in \mathcal{B}_w(\mu) \sqcup \{ 0 \}, \\ 1 \leq b \leq \langle \wt(\pi_2), \alpha_i^\vee \rangle. \end{array} \right. \right\} = \emptyset.
\end{align*}
This proves (\ref{RHS_2}).

From now on, we assume that $\ell(s_i w) > \ell(w)$. In this case, the set $\bigcup_{n \geq 0}\tilde{f}_i^n(\mathcal{B}_w(\mu)) \setminus \{ 0 \}$ is identical to $\mathcal{B}_{s_i w}(\mu)$. Hence if we take $\pi_2 \in \mathcal{B}_w(\mu)$ and $b \in \mathbb{Z}$ such that $\tilde{f}_{i}(\pi_2) \not\in \mathcal{B}_w(\mu) \sqcup \{ 0 \}$ and $1 \leq b \leq \langle \wt(\pi_2), \alpha_i^\vee \rangle$, then we have $\tilde{f}_i^b(\pi_2) \in \mathcal{B}_{s_i w}(\mu)$. Moreover, the element $\tilde{f}_i^b(\pi_2)$ is not contained in $\mathcal{B}_w(\mu)$ since $\tilde{f}_{i}(\pi_2) \not\in \mathcal{B}_w(\mu) \sqcup \{ 0 \}$. Hence $\tilde{f}_i^b(\pi_2) \in \mathcal{B}_{s_i w}(\mu) \setminus \mathcal{B}_w(\mu)$.

For the opposite inclusion, we take an arbitrary element $\xi \in \mathcal{B}_{s_i w}(\mu) \setminus \mathcal{B}_w(\mu)$, set $\pi_2 := \tilde{e}_i^{\max}(\xi)$, and take $b \in \mathbb{Z}$ such that $\xi = \tilde{f}_i^b(\pi_2)$. By definition of $\pi_2$, we have $\tilde{e}_i(\pi_2) = 0$, and hence it follows that $\varepsilon_i(\pi_2) = 0$ and $\varphi_i(\pi_2) = \langle \wt(\pi_2), \alpha_i^\vee \rangle$. We claim that $\pi_2 \in \mathcal{B}_w(\mu)$. To check this, take a reduced expression $w = s_{j_1} \cdots s_{j_t}$ for $w$. Since $\ell(s_i w) > \ell(w)$ by our assumption, $s_i w = s_i s_{j_1} \cdots s_{j_t}$ is a reduced expression for $s_i w$. Take a string parametrization $\xi = \tilde{f}_{i}^{a_0} \tilde{f}_{j_1}^{a_1} \cdots \tilde{f}_{j_t}^{a_t} (\pi^\mu)$ for $\xi$. Then we have $\pi_2 = \tilde{e}_i^{\max}(\xi) = \tilde{f}_{j_1}^{a_1} \cdots \tilde{f}_{j_t}^{a_t} (\pi^\mu) \in \mathcal{B}_w(\mu)$. Also, we see that $b \geq 1$ since $\pi_2 \in \mathcal{B}_w(\mu)$ and $\xi = \tilde{f}_i^b(\pi_2) \not\in \mathcal{B}_w(\mu)$. Hence $\tilde{f}_i(\pi_2) \not= 0$. Suppose that $\tilde{f}_i(\pi_2) \in \mathcal{B}_w(\mu)$, then $\xi$ must be contained in $\mathcal{B}_w(\mu)$ because of the string property for $\mathcal{B}_w(\mu)$. This contradicts the assumption that $\xi \not\in \mathcal{B}_w(\mu)$. Hence $\tilde{f}_i(\pi_2) \not\in \mathcal{B}_w(\mu)$. Finally, one has $b \leq \varphi_i(\pi_2) = \langle \wt(\pi_2), \alpha_i^\vee \rangle$ since $\tilde{f}_i^b(\pi_2) \not= 0$. From these, we deduce that
\begin{align*}
\xi \in \left\{ \tilde{f}_{i}^b(\pi_2) \left| \begin{array}{l}\pi_2 \in \mathcal{B}_w(\mu), \\ \tilde{f}_{i}(\pi_2) \not\in \mathcal{B}_w(\mu) \sqcup \{ 0 \}, \\ 1 \leq b \leq \langle \wt(\pi_2), \alpha_i^\vee \rangle. \end{array} \right. \right\}.
\end{align*}

The above argument shows (\ref{RHS_2}). Thus, the proof of (\ref{to_prove_2}) has been completed.
\end{proof}


\begin{thebibliography}{99}
\bibitem[F]{Fulton} W. Fulton, \textit{Young Tableaux: With Applications to Representation Theory and Geometry}, London Mathematical Society Student Texts, vol. 35, Cambridge University Press, Cambridge, 1997. 
\bibitem[HK]{Hong} J. Hong and S. J. Kang, \textit{Introduction to Quantum Groups and Crystal Bases}, Graduate Studies in Mathematics, vol. 42, American Mathematical Society, Providence, RI, 2002.
\bibitem[HLMvW]{Haglund} J. Haglund, K. Luoto, S. Mason, and S. van Willigenburg, \textit{Refinements of the Littlewood-Richardson rule}, Trans. Amer. Math. Soc. \textbf{363} (2011), 1665--1686.
\bibitem[Hu]{Humphreys} J. E. Humphreys, \textit{Reflection Groups and Coxeter Groups}, Cambridge Studies in Advanced Mathematics, vol. 29, Cambridge University Press, New York, 1990.
\bibitem[Ja]{Jantzen} J. C. Jantzen, \textit{Lectures on Quantum Groups}, Graduate Studies in Mathematics, vol. 6, American Mathematical Society, Providence, RI, 1996.
\bibitem[Jo]{Joseph} A. Joseph, \textit{Quantum Groups and Their Primitive Ideals}, Ergebnisse der Mathematik und ihrer Grenzgebiete, vol. 29, Springer–Verlag, Berlin, 1995.
\bibitem[Ka1]{Kashiwara} M. Kashiwara, \textit{The crystal base and Littelmann's refined Demazure character formula}, Duke Math. J. \textbf{71} (1993), no. 3, 839--858.
\bibitem[Ka2]{Kashiwara2} M. Kashiwara, \textit{Similarity of crystal bases}, Contemp. Math. \textbf{194} (1993), 177--186.
\bibitem[LLM]{Magyar} V. Lakshmibai, P. Littelmann, and P. Magyar, \textit{Standard monomial theory for Bott-Samelson varieties}, Compos. Math. \textbf{130} (2002), no. 3, 293--318. 
\bibitem[Li1]{Littelmann_LR} P. Littelmann, \textit{A Littlewood-Richardson rule for symmetrizable Kac-Moody algebras}, Invent. Math. \textbf{116} (1994), no. 1, 329--346.\bibitem[Li2]{Littelmann_path} P. Littelmann, \textit{Paths and root operators in representation theory}, Ann. of Math. \textbf{142} (1995), no. 3, 499--525.
\bibitem[Li3]{Littelmann} P. Littelmann, \textit{Cones, crystals, and patterns}, Transform. Groups \textbf{3} (1998), no. 2, 145--179.
\bibitem[RS]{Reiner} V. Reiner and M. Shimozono, \textit{Key polynomials and a flagged Littlewood-Richardson rule}, J. Combin. Theory Ser. A, \textbf{70} (1995), no. 1, 107--143
\end{thebibliography}
\end{document}